\documentclass[11pt]{amsart}
\usepackage[utf8]{inputenc}
\usepackage{amsmath,amsfonts}
\usepackage{amssymb}
\usepackage{amscd}
\usepackage{amsthm}
\usepackage{yhmath}
\usepackage{subfigure}

\usepackage[all]{xy}
\usepackage{color}

\usepackage{pdfsync}
\usepackage[colorlinks = true, pagebackref]{hyperref}
\usepackage{color}
\usepackage{pgf,tikz}
\usetikzlibrary{arrows}
\usepackage{cancel}
\usepackage{color}
\usepackage{subfigure}
\usepackage{float}
\usepackage{verbatim}
\usepackage{wrapfig}
\usepackage{mathtools}

\usepackage[margin=1.26in]{geometry}

\usepackage[normalem]{ulem}

\usepackage{pdfsync}
\usepackage{color}
\usepackage{mathrsfs}

\usepackage[T1]{fontenc}
\usepackage{mathrsfs}
\usepackage{epsfig}
\usepackage{xypic}
\usepackage{verbatim}
\usepackage{multicol}
\usepackage{wrapfig}
\usepackage{pgf,tikz}
\usepackage{mathrsfs}
\usetikzlibrary{arrows}

\numberwithin{equation}{section}

	\usepackage{amsrefs}

\newcommand{\N}{\mathbb{N}}

\newcommand{\R}{\mathbb{R}}

\usepackage{mathrsfs}
\usepackage{mathtools}
\usepackage{comment}
\usepackage{indentfirst}
\usepackage{braket}
\usepackage{scalerel}

{\left\lbrace\begin{array}{@{}l@{}}}%
{\end{array}\right.}

\makeatletter
\newcommand*{\mint}[1]{%
	\mint@l{#1}{}%
}
\newcommand*{\mint@l}[2]{%
	\@ifnextchar\limits{%
		\mint@l{#1}%
	}{%
	\@ifnextchar\nolimits{%
		\mint@l{#1}%
	}{%
	\@ifnextchar\displaylimits{%
		\mint@l{#1}%
	}{%
	\mint@s{#2}{#1}%
}%
}%
}%
}
\newcommand*{\mint@s}[2]{%
	\@ifnextchar_{%
		\mint@sub{#1}{#2}%
	}{%
	\@ifnextchar^{%
		\mint@sup{#1}{#2}%
	}{%
	\mint@{#1}{#2}{}{}%
}%
}%
}
\def\mint@sub#1#2_#3{%
	\@ifnextchar^{%
		\mint@sub@sup{#1}{#2}{#3}%
	}{%
	\mint@{#1}{#2}{#3}{}%
}%
}
\def\mint@sup#1#2^#3{%
	\@ifnextchar_{%
		\mint@sub@sup{#1}{#2}{#3}%
	}{%
	\mint@{#1}{#2}{}{#3}%
}%
}
\def\mint@sub@sup#1#2#3^#4{%
	\mint@{#1}{#2}{#3}{#4}%
}
\def\mint@sup@sub#1#2#3_#4{%
	\mint@{#1}{#2}{#4}{#3}%
}
\newcommand*{\mint@}[4]{%
	\mathop{}%
	\mkern-\thinmuskip
	\mathchoice{%
		\mint@@{#1}{#2}{#3}{#4}%
		\displaystyle\textstyle\scriptstyle
	}{%
	\mint@@{#1}{#2}{#3}{#4}%
	\textstyle\scriptstyle\scriptstyle
}{%
\mint@@{#1}{#2}{#3}{#4}%
\scriptstyle\scriptscriptstyle\scriptscriptstyle
}{%
\mint@@{#1}{#2}{#3}{#4}%
\scriptscriptstyle\scriptscriptstyle\scriptscriptstyle
}%
\mkern-\thinmuskip
\int#1%
\ifx\\#3\\\else_{#3}\fi
\ifx\\#4\\\else^{#4}\fi  
}
\newcommand*{\mint@@}[7]{%
	\begingroup
	\sbox0{$#5\int\m@th$}%
	\sbox2{$#5\int_{}\m@th$}%
	\dimen2=\wd0 %
	\let\mint@limits=#1\relax
	\ifx\mint@limits\relax
	\sbox4{$#5\int_{\kern1sp}^{\kern1sp}\m@th$}%
	\ifdim\wd4>\wd2 %
	\let\mint@limits=\nolimits
	\else
	\let\mint@limits=\limits
	\fi
	\fi
	\ifx\mint@limits\displaylimits
	\ifx#5\displaystyle
	\let\mint@limits=\limits
	\fi
	\fi
	\ifx\mint@limits\limits
	\sbox0{$#7#3\m@th$}%
	\sbox2{$#7#4\m@th$}%
	\ifdim\wd0>\dimen2 %
	\dimen2=\wd0 %
	\fi
	\ifdim\wd2>\dimen2 %
	\dimen2=\wd2 %
	\fi
	\fi
	\rlap{%
		$#5%
		\vcenter{%
			\hbox to\dimen2{%
				\hss
				$#6{#2}\m@th$%
				\hss
			}%
		}%
		$%
	}%
	\endgroup
}

\makeatletter
\def\overbracket#1{\mathop{\vbox{\ialign{##\crcr\noalign{\kern3\p@}
				\downbracketfill\crcr\noalign{\kern3\p@\nointerlineskip}
				$\hfil\displaystyle{#1}\hfil$\crcr}}}\limits}
\def\underbracket#1{\mathop{\vtop{\ialign{##\crcr
				$\hfil\displaystyle{#1}\hfil$\crcr\noalign{\kern3\p@\nointerlineskip}
				\upbracketfill\crcr\noalign{\kern3\p@}}}}\limits}
\def\overparenthesis#1{\mathop{\vbox{\ialign{##\crcr\noalign{\kern3\p@}
				\downparenthfill\crcr\noalign{\kern3\p@\nointerlineskip}
				$\hfil\displaystyle{#1}\hfil$\crcr}}}\limits}
\def\underparenthesis#1{\mathop{\vtop{\ialign{##\crcr
				$\hfil\displaystyle{#1}\hfil$\crcr\noalign{\kern3\p@\nointerlineskip}
				\upparenthfill\crcr\noalign{\kern3\p@}}}}\limits}
\def\downparenthfill{$\m@th\braceld\leaders\vrule\hfill\bracerd$}
\def\upparenthfill{$\m@th\bracelu\leaders\vrule\hfill\braceru$}
\def\upbracketfill{$\m@th\makesm@sh{\llap{\vrule\@height3\p@\@width.7\p@}}%
	\leaders\vrule\@height.7\p@\hfill
	\makesm@sh{\rlap{\vrule\@height3\p@\@width.7\p@}}$}
\def\downbracketfill{$\m@th
	\makesm@sh{\llap{\vrule\@height.7\p@\@depth2.3\p@\@width.7\p@}}%
	\leaders\vrule\@height.7\p@\hfill
	\makesm@sh{\rlap{\vrule\@height.7\p@\@depth2.3\p@\@width.7\p@}}$}
\makeatother

\theoremstyle{theorem}
\newtheorem{theorem}{\sc \textbf{Theorem}}[section]  
\newtheorem{proposition}[theorem]{\sc \textbf{Proposition}}   
\newtheorem{corollary}[theorem]{\sc \textbf{Corollary}}        
\newtheorem{lemma}[theorem]{\sc \textbf{Lemma}}

\theoremstyle{remark}
\newtheorem{definition}[theorem]{\sc Definition}

\usepackage{amsfonts}

\newcommand{\loc}{\mathrm{loc}}

\newcommand{\dd}{\mathrm{d}}

\title[Sobolev spaces on Lie groups]{Sobolev spaces on Lie groups:\\ embedding theorems and algebra properties} 

\author[Bruno]{Tommaso Bruno}
\address{Dipartimento di Scienze Matematiche ``Giuseppe Luigi Lagrange'',
  Politecnico di Torino, Corso Duca degli Abruzzi 24, 10129 Torino,
  Italy - Dipartimento di Eccellenza 2018-2022}
\email{tommaso.bruno@polito.it}

\author[Peloso]{Marco M.\ Peloso}
\address{Dipartimento di Matematica, 
Universit\`a degli Studi di Milano, 
Via C.\ Saldini 50,  
20133 Milano, Italy}
\email{marco.peloso@unimi.it}

\author[Tabacco]{Anita Tabacco}
\author[Vallarino]{Maria Vallarino}
\address{Dipartimento di Scienze Matematiche ``Giuseppe Luigi Lagrange'',
  Politecnico di Torino, Corso Duca degli Abruzzi 24, 10129 Torino,
  Italy - Dipartimento di Eccellenza 2018-2022}
\email{anita.tabacco@polito.it}
\email{maria.vallarino@polito.it}


\thanks{All authors are  partially supported by the grant PRIN 2015 {\em Real and Complex Manifolds: Geometry, Topology and Harmonic Analysis}, and are members of the Gruppo Nazionale per l'Analisi Matematica, la Probabilit\`a e le loro Applicazioni (GNAMPA) of the Istituto Nazionale di Alta Matematica (INdAM)}

\date{}

\begin{document}
	\maketitle
	\begin{abstract}
Let $G$ be a noncompact connected Lie group, denote with $\rho$ a right Haar measure and choose a family of linearly independent left-invariant vector fields $\mathbf{X}$ on $G$ satisfying H\"ormander's condition. Let $\chi$ be a positive character of $G$ and consider the measure $\mu_\chi$ whose density  with respect to $\rho$ is $\chi$. In this paper, we introduce Sobolev spaces $L^p_\alpha(\mu_\chi)$ adapted to $\mathbf{X}$ and $\mu_\chi$ ($1<p<\infty$, $\alpha\geq 0$) and study embedding theorems and algebra properties of these spaces. As an application, we prove local well-posedness and regularity results of solutions of some nonlinear heat and Schr\"odinger equations on the group. 

	\end{abstract}
	
	\section{Introduction}
In order to study differential equations on a smooth manifold and the properties of their solutions, one needs a way to measure the regularity of functions.  If the manifold is endowed with a Riemannian structure, it has been proven along the years that Sobolev norms defined in terms of the Laplace--Beltrami operator are very natural tools.  In such setting, the properties of the homogeneous and non-homogeneous Sobolev spaces, as well as of those of the Besov spaces -- embedding theorems and algebra properties, for istance -- are well understood; see~\cite{Hebey,Aubin,Robinson} e.g., and references therein. On a  sub-Riemannian manifold, though still very natural,  these types of questions are much less  understood.  In the fundamental works of Folland and Stein~\cite{FollandStein} and Folland~\cite{Folland} on the Heisenberg and stratified nilpotent Lie groups, respectively, the authors studied the inverse of a certain class of operators including the standard ``sum-of-squares'' sub-Laplacian, and introduced Sobolev spaces defined in terms of fractional powers of the sub-Laplacian itself.  Their results resembled the classical results known in the Riemannian and elliptic case, thus completely justifying the introduction of these spaces. Since then, this approach was shown to provide satisfying results by many authors, who successfully generalised Folland and Stein's construction in various settings, see~\cite{Bohnke, CRTN, PV, TER1, Varopoulos1} e.g. 

\smallskip

In the remarkable paper ~\cite{ABGR}, Agrachev, Boscain, Gauthier and Rossi introduced an \emph{intrinsic hypoelliptic Laplacian} on any regular sub-Riemannian manifold, generalizing the Laplace--Beltrami operator on a Riemannian manifold. On a Lie group endowed with the Carnot--Carathéodory metric induced by a family of left-invariant vector fields satisfying H\"ormander's condition, this operator  is a sub-Laplacian with a first-order drift term. This evokes an interesting paper of Hebisch, Mauceri and Meda~\cite{HMM}, where the authors established a one-to-one correspondence between sub-Laplacians with drift on Lie groups that are symmetric on some $L^2$-spaces, and those measures whose density is a continuous positive character with respect to a right Haar measure of the group. Thus, we are led to consider noncompact, connected Lie groups and study Sobolev  spaces defined in terms of sub-Laplacians with a drift and endowed with such measures, and we do so in this paper, as we now illustrate.

\smallskip

Let $G$ be a noncompact connected Lie group with identity $e$ and let $\mathbf{X}= \{X_1, \dots ,X_\ell \}$ be a family of linearly independent, left-invariant vector fields of $G$ satisfying H\"ormander's condition. Denote with $\rho$ a right Haar measure of $G$, let $\chi$ be a continuous positive character of $G$, and consider the measure $\mu_\chi$ on $G$ whose density is $\chi$ with respect to $\rho$, that is 
\begin{equation}\label{muchi}
\dd \mu_\chi = \chi \, \dd \rho.
\end{equation}
Let $c_j=(X_j \chi)(e)$, $j=1,\dots, \ell$, and consider the differential operator 
\[
\Delta_\chi = - \sum_{j=1}^\ell (X_j^2 + c_jX_j)
\]
with domain the set of smooth and compactly supported functions $C_c^\infty(G)$ on $G$. As mentioned above, these operators were introduced by Hebisch, Mauceri and Meda~\cite{HMM}. On the one hand, they proved that $\Delta_\chi$ is essentially self-adjoint on $L^2(\mu_\chi)$. On the other hand, they proved that if a sub-Laplacian with drift is symmetric on $L^2(\mu)$ for a positive measure $\mu$ on $G$, then $\mu=\mu_\chi$ for a positive character $\chi$ of $G$ and the drift has the form above. 

With a slight abuse of notation, we still denote with $\Delta_{\chi}$ the smallest closed extension of $\Delta_\chi$ on $L^p(\mu_\chi)$, $1<p<\infty$, and define the Sobolev spaces 
\[
L^p_\alpha (\mu_\chi) \coloneqq \big\{ f\in L^p(\mu_\chi) \colon
\Delta_{\chi}^{\alpha/2} f\in L^p(\mu_\chi) \big\} 
\]
endowed with the norm
\[
\| f\|_{L^p_\alpha (\mu_\chi)} \coloneqq \|f\|_{L^p(\mu_\chi)} + \|\Delta_{\chi}^{\alpha/2} f\|_{L^p(\mu_\chi)}.
\]
We remark that in the unimodular case and with $\chi=1$ these spaces coincide with the Sobolev spaces defined by the classical left-invariant sum-of-squares operator $\Delta= -\sum_{j=1}^\ell X_j^2$ and studied extensively in~\cite{CRTN}. In the nonunimodular case, the spaces $L^p_\alpha (\mu_\chi)$ are a generalization of those considered in~\cite{PV}.  Furthermore, we observe  that the  operator $\Delta= - \sum_{j=1}^\ell X_j^2$ is not even symmetric on $L^2(\mu_\chi)$ if $\chi\neq 1$, so that a Sobolev space adapted to the measure $\mu_\chi$ when $\chi\neq 1$ cannot be defined by means of fractional powers of $\Delta$.  

The spaces $L^p_\alpha(\mu_\chi)$ emerge indeed as the natural Sobolev spaces adapted to the measure $\mu_\chi$, for if $\alpha=k$ is an integer, then 

\begin{equation}\label{carinteger}
\|f\|_{L^p_k(\mu_\chi)}\approx \sum_{ J \in \{1,\dots, \ell \}^m,\, m\leq k} \| X_J f\|_{L^p(\mu_\chi)},
\end{equation}
where $X_J = X_{j_1}\cdots X_{j_m}$ and  $J= (j_1,\dots, j_m)$. The characterization~\eqref{carinteger} of the Sobolev norm is a consequence of the boundedness on $L^p(\mu_\chi)$, $1<p<\infty$, of the local Riesz transforms of any order associated with $\Delta_\chi$ (Theorem~\ref{teoLocalRiesz}), which we obtain by means of endpoint results involving suitable Hardy and BMO spaces and by interpolation. 

\smallskip

The aim of this paper is to investigate the structure of the spaces $L^p_\alpha (\mu_\chi)$, their algebra properties  and embedding theorems. We also apply these results to some nonlinear Cauchy problems associated with $\Delta_\chi$. When $\chi=1$ and the group $G$ is unimodular, embedding theorems for these Sobolev spaces  are known in the stratified case \cite{Folland}, graded case \cite{Fischer-Ruzhanski}, and may be considered as folklore in the case of Lie groups of polynomial growth. However, we could not find any precise reference in the literature. 

On nonunimodular groups, only the case of first-order Sobolev spaces was studied by Varopoulos~\cite{Varopoulos1} when $\chi$ is a power of the modular function $\delta$. Again when $\chi=1$, Coulhon, Russ and Tardivel-Nachef~\cite{CRTN} investigated algebra properties of Sobolev spaces on unimodular Lie groups, while in the nonunimodular case they were proved in~\cite{PV} by the second and fourth named authors of the present paper. In \cite{PV} it is also proved that the embedding $L^p_\alpha(\rho) \hookrightarrow L^\infty$ cannot hold, in general, if the group is nonunimodular. By~\eqref{carinteger}, however, this is not surprising. Indeed, an easy translation-invariance argument (see Section~\ref{Sec_SET}) shows that an embedding of the form  
\begin{equation*}
L^p_\alpha(\mu_\chi) \hookrightarrow L^q(\mu_\chi), \quad p\in (1,\infty), \; \alpha\geq 0, \quad q\in (1,\infty]\setminus \{p\} 
\end{equation*}
for some positive character $\chi$ may hold only if $\chi$ is the modular function $\delta$, i.e.\ $\mu_\chi =\lambda$ is the left Haar measure of $G$. In other words, the natural space to look for Sobolev embeddings turns out to be $L^p_\alpha(\lambda)$, defined as shown above in terms of $\Delta_\delta$. Under this point of view, it is interesting to notice that $\Delta_\delta$, which from now on will be denoted with $\mathcal{L}$, is the aforementioned intrinsic hypoelliptic Laplacian associated with the Carnot--Carathéodory metric induced by ${\bf X}$ which was introduced in~\cite{ABGR}. We recall that $\mathcal{L}$ reduces to the sum-of-squares operator $\Delta$ if and only if $G$ is unimodular. 

The two main results of the paper are as follows. First, we prove an embedding theorem for the Sobolev spaces $L^p_\alpha(\mu_\chi)$ for any choice of the positive character $\chi$ of the group. More precisely, we have the following:
\begin{theorem}\label{SETright}
Let $\chi$ be a positive character of $G$ and $1< p,q <\infty$.
\begin{itemize}
\item[(a)] If $\alpha>0$, then $L^p_\alpha(\mu_\chi) \hookrightarrow L^q(\mu_{\chi^{q/p}\delta^{1-q/p}})$ for every $q\geq p$ such that  $\frac1p-\frac1q \le \frac\alpha d$; 
\item[(b)] if $\alpha\geq d/p$, then $L^p_\alpha(\mu_\chi) \hookrightarrow L^q(\mu_{\chi^{q/p}\delta^{1-q/p}})$ for every $q\geq  p$; 
\item[(c)] if $\alpha> d/p$, then $L^p_\alpha(\mu_\chi) \hookrightarrow (\delta \chi^{-1})^{1/p} L^{\infty}$. 
\end{itemize}
\end{theorem}
In the statement above, the positive constant $d$ is the local dimension of $G$ as defined below in~\eqref{pallepiccole} and, if $\kappa$ is a positive character of the group, $\kappa L^{\infty}$ stands for the Banach space of functions $\{ \kappa f\colon   f\in L^\infty \}$ endowed with the norm $\| \kappa f \|_{ \kappa L^\infty } \coloneqq \| f\|_{\infty}$.

In the special case when $\chi=\delta$, Theorem~\ref{SETright} provides embedding results between spaces endowed with the same measure, which is the left Haar measure (see Theorem~\ref{SET}); as mentioned before, this is the unique case for which such embeddings hold true.

As a second main result, we prove algebra properties of the spaces $L^p_\alpha(\mu_\chi)$: 
\begin{theorem}\label{algebraproduct}
Let $\chi$ be a positive character of $G$, $\alpha\geq 0$, $p_1,q_2\in(1,\infty]$ and $p,p_2,q_1\in (1,\infty)$ such that $\frac1p=\frac{1}{p_i}+\frac{1}{q_i}$, $i=1,2$. Then, for all $f\in L^{p_1}(\mu_\chi)\cap L^{p_2}_{\alpha}(\mu_\chi)$ and $g\in L^{q_2}(\mu_\chi)\cap L^{q_1}_{\alpha}(\mu_\chi)$, we have $fg\in L^p_{\alpha}(\mu_\chi)$ and 
\[
\|fg\|_{L^p_\alpha(\mu_\chi)}\lesssim
\|f\|_{L^{p_1}(\mu_\chi)}\|g\|_{L^{q_1}_{\alpha}(\mu_\chi)}+\|f\|_{L^{p_2}_{\alpha}(\mu_\chi)}\|g\|_{L^{q_2}(\mu_\chi)}. 
\]
In particular, $L^p_\alpha(\mu_\chi)\cap L^\infty$ is an algebra for every $p\in (1,\infty)$.
\end{theorem} 
As a consequence of Theorems~\ref{SETright} and~\ref{algebraproduct}, we obtain that $L^p_\alpha(\lambda)$ is an algebra for every $p\in (1,\infty)$ and $\alpha>d/p$, thus extending the classical results by Strichartz~\cite{Strichartz}  and Bohnke~\cite{Bohnke}.  We also prove that if $\chi \neq \delta$, i.e.\ if $\mu_\chi$ is not the left measure, then $L^p_\alpha(\mu_\chi)$ is not an algebra, in general, and it does not embed in $L^\infty$. By means of the algebra property of $L^p_\alpha(\mu_\chi)\cap L^\infty$, we finally prove local well-posedness and regularity results of solutions of some nonlinear heat equations associated with $\Delta_\chi$, and some nonlinear Schr\"odinger equations associated with $\mathcal{L}$. 

\smallskip

The content of the paper is as follows. In Section~\ref{Sec2} we introduce the operator $\Delta_\chi$ and prove some properties of its heat kernel $p_t^\chi$. In Section~\ref{Sec_Sobolev} we introduce the spaces $L^p_\alpha(\mu_\chi)$, study their interpolation properties (Lemma~\ref{lemma_interpolation}) and characterize their norms either when the regularity $\alpha$ is an integer (Proposition~\ref{prop:kintero}) or by means of a recursive formula (Proposition~\ref{alpha+1}). In Section~\ref{Sec_SET}, we prove the embedding results of Theorem~\ref{SETright}. In Section~\ref{Sec_algebra} we prove Theorem~\ref{algebraproduct} and provide some counterexamples in the case of the affine group of the real line. Section~\ref{Sec_applications} is devoted to applications to nonlinear problems, and Section~\ref{Sec_Proof_Riesz} to the proof of the $L^p(\mu_\chi)$-boundedness of the local Riesz transforms associated with $\Delta_\chi$ (Theorem~\ref{teoLocalRiesz}). Section~\ref{Sec:finrem} contains some final remarks.

\section{Lie groups and the sub-Laplacian $\Delta_\chi$}\label{Sec2} 

\subsection{Preliminaries on Lie groups}
All throughout the paper, $G$ will be a noncompact, connected Lie group with identity $e$. We shall denote with $\rho$ a right Haar measure, with $\lambda$ a left Haar measure and with $\delta$ the modular function, i.e.\ the function on $G$ such that 
\[
\dd \lambda = \delta \, \dd \rho.
\]
It is well known that $\delta$ is a smooth positive character of $G$, i.e.\ a smooth homomorphism of $G$ into the multiplicative group $\R^+$. The letter $\chi$ will always denote a continuous positive character of $G$, which is then automatically smooth. We shall denote with $\mu_\chi$ the measure with density $\chi$ with respect to $\rho$ as in~\eqref{muchi}. Observe that $\mu_\delta=\lambda$ and $\mu_1=\rho$.

The set $\mathbf{X}= \{ X_1,\dots, X_\ell\}$ will be a family of left-invariant, linearly independent vector fields which satisfy H\"ormander's condition. These vector fields induce a left-invariant distance $d_C(\cdot, \cdot)$ which is called Carnot--Carathéodory distance. Given a ball $B$ with respect to such distance, $c_B$ will denote its center and $r_B$ its radius. We also write $B=B(c_B,r_B)$. We write $|x|=d_C(x,e)$, and $B_r=B(e,r)$. The volume of the ball $B_r$ with respect to the right Haar measure $\rho$ will be denoted with $V(r)=\rho(B_r)$.  

It is well known (cf.~\cite{Guiv,Varopoulos1}) that there exist two constants, which we call $d=d(G, \mathbf{X})$ and $D=D(G)$, such that 
\begin{equation}\label{pallepiccole}
V(r)\approx r^d\qquad \forall r\in (0,1]
\end{equation}
and
\begin{equation}\label{pallegrandi}
V(r)\lesssim e^{Dr}\qquad \forall r\in (1,\infty).
\end{equation}
They are usually called \emph{local} and \emph{global} dimensions of the metric measure space $(G, d_C, \rho)$, respectively. As a consequence of~\eqref{pallepiccole}, the space $(G, d_C, \rho)$ is locally doubling. Since for every character $\chi$ and $R>0$, there exists a constant $c=c(\chi, R)$ such that 
\begin{equation}\label{localdoubmugamma}
c^{-1} \chi(x) \leq \chi (y) \leq c \chi(x) \qquad \forall \, x,y\in G
\, \mbox{ s.t. }\, d_C(x,y)\leq R 
\end{equation}
then also the metric measure space $(G, d_C, \mu_\chi)$ is locally doubling.

If $p\in (1,\infty)$, the spaces of (equivalent classes of) measurable functions whose $p$-power is integrable with respect to $\mu_\chi$ will be denoted by $L^p(\mu_\chi)$, and endowed with the usual norm which we shall denote with $\| \cdot \|_{L^p(\mu_\chi)}$. The space $L^\infty$ will be the space of (equivalent classes of) measurable functions which are $\rho$-essentially bounded; observe that this coincides with the space of $\mu_\chi$-essentially bounded functions for every positive character $\chi$ of $G$, for $\mu_\chi$ is absolutely continuous with respect to $\rho$. 

We shall denote with $\Delta$ the smallest self-adjoint extension on $L^2(\rho)$ of the ``sum-of-squares'' operator  
\[
\Delta \coloneqq - \sum_{j=1}^\ell X_j^2
\]
on $C_c^\infty(G)$. The smooth integral kernel of $e^{-t\Delta}$ will be denoted with $P_t(\cdot, \cdot )$, and its smooth convolution kernel with $p_t$, i.e.\ $e^{-t\Delta} f = f* p_t$ where the convolution between two functions $f$ and $g$, when it exists, is  
	\[
	f*g(x) =\int_G f(xy^{-1})g(y)\, \dd \rho(y).
	\]
They are related to each other by the relation
\begin{equation}\label{Pt_pt}
P_t(x,y)= p_t(y^{-1}x)\delta(y) \qquad \forall \, x,y\in G.
\end{equation}
All throughout the paper, we shall denote with $\mathfrak{l}$ the set $\{1,\dots, \ell \}$. For every $m\in\N$, $\mathfrak{l}^m$ will be the set of multi-indices $J=(j_1,\dots ,j_m)$ such that $j_i \in \mathfrak{l}$ and for $J\in \mathfrak{l}^m$ we denote by $X_J$ the left-invariant differential operator
\[
X_J = X_{j_1}\cdots X_{j_m} \,.
\]
For any quantity $A$ and $B$, we shall write $A\lesssim B$ by meaning that there exists a constant $c>0$ such that $A\leq c \,B$. If $A\lesssim B$ and $B\lesssim A$, we write $A\approx B$ . 

\subsection{Hardy and BMO spaces}\label{Subsec_HardyBMO}
We now introduce the local Goldberg type atomic Hardy space $\mathfrak h^1(\mu_\chi)$ and its dual $\mathfrak{bmo}(\mu_\chi)$ adapted to the measure $\mu_\chi$. The space $\mathfrak{h}^1(\mu_\chi)$ is the analog in the metric measure space $(G,d_C, \mu_\chi)$ of the Hardy space defined by Goldberg in the Euclidean setting~\cite{G}. They were introduced by Meda and Volpi~\cite{MedaVolpi} and Taylor~\cite{T} in more general contexts. Here and in the following, we denote by $\mathcal{B}_s$ the family of balls of radius $\leq s$. 

\begin{definition}\label{def:Hardy}
A standard $\mathfrak{h}^1$-atom is a function $a$ in $L^1(\mu_\chi)$ supported in a ball $B\in \mathcal{B}_1$ such that 
\begin{itemize}
\item $\|a\|_{L^2(\mu_\chi)}\leq \mu_\chi(B)^{-1/2}$, and
\item $\int a \,\dd\mu_\chi=0$.
\end{itemize}
A global $\mathfrak{h}^1$-atom is a function $a$ in $L^1(\mu_\chi)$ supported in a ball $B$ of radius $1$ such that $\|a\|_{L^2(\mu_\chi)}\leq \mu_\chi(B)^{-1/2}$. We shall refer to standard and global $\mathfrak{h}^1$-atoms only as \emph{admissible} atoms.   

The space $\mathfrak{h}^1(\mu_\chi)$ is the space of functions $f$ in $L^1(\mu_\chi)$ such that $f=\sum_j c_ja_j$, where $(c_j)\in \ell^1$ and $a_j$ are admissible atoms. The norm $\|f\|_{\mathfrak{h}^1(\mu_\chi)}$ is defined as the infimum of $\|(c_j)\|_{\ell^1}$ over all atomic decompositions of $f$. 
\end{definition}

\begin{definition}
The space $\mathfrak{bmo}(\mu_\chi)$ is the space of all equivalence classes of locally integrable functions $g$ modulo constants such that  
\[
\begin{aligned}
\|g\|_{\mathfrak{bmo}(\mu_\chi)} \coloneqq \sup_{B\in\mathcal B_1} \left( \frac{1}{\mu_\chi(B)} \int_B |g-g_B|^2 \, \dd\mu_\chi \right)^{1/2} +  \sup_{x\in G} \left( \frac{1}{\mu_\chi(B(x,1))}  \int_{B(x,1)} |g|^2\,  \dd\mu_\chi \right)^{1/2} 
\end{aligned}
\]
is finite, where $g_B = \mu_\chi(B)^{-1} \int_B g \, \dd \mu_\chi$.
\end{definition}
By~\cite[Theorem 2]{MedaVolpi} the dual of $\mathfrak h^1(\mu_\chi)$ can be identified with $\mathfrak{bmo}(\mu_\chi)$. We also recall that, if one considers admissible atoms as in Definition~\ref{def:Hardy} supported in balls in $\mathcal{B}_s$ for any fixed positive $s$, then the corresponding atomic Hardy space coincides with $\mathfrak{h}^1(\mu_\chi)$ and has an equivalent norm (cf.~\cite[Proposition 1]{MedaVolpi}). Moreover, the spaces $L^p(\mu_\chi)$ when $p\in(1,2)$ (resp.\ $p\in(2,\infty)$) are the complex interpolation spaces between $\mathfrak{h}^1(\mu_\chi)$ (resp.\ $\mathfrak{bmo}(\mu_\chi)$) and $L^2(\mu_\chi)$, as shown in~\cite[Theorem 5]{MedaVolpi}.

\subsection{The weighted sub-Laplacian $\Delta_{\chi}$}	\label{Sec_Deltachi} 
Let $\chi$ be a positive character of $G$ and let $c_i =(X_i \chi)(e)$, $i\in \mathfrak{l}$. It is easy to see (cf.~\cite[p.\ 124]{VSCC}) that 
\begin{equation}\label{Xdelta}
X_i \chi = c_i \chi  \qquad i=1,\dots, \ell.
\end{equation}
Starting from~\eqref{Xdelta}, a straightforward computation shows that the formal adjoint of $X_i$ on $L^2(\mu_\chi)$ is $X^{*,\chi}_i =-X_i - c_i$. Therefore, the operator 
\[
\sum_{i=1}^\ell X_i^{*,\chi }X_i = \Delta -\sum_{i=1}^\ell c_i X_i
\]
with domain $C_c^\infty(G)$ is symmetric and densely defined on $L^{2}(\mu_\chi)$. Thus it is closable, and we denote by $\Delta_\chi$ its closure. We shall denote with $X$  the drift term $\sum_{i=1}^\ell c_i X_i$.

If we consider (see also~\cite{HMM}) the unitary operator
\[
\mathcal{U}_2 \colon L^{2}(\mu_\chi) \to L^2(\rho), \qquad \mathcal{U}_2 f \coloneqq f\chi^{1/2},
\]
then a simple computation shows that
\begin{equation}\label{unitary_equivalence}
\mathcal{U}_2 \Delta_\chi  \mathcal{U}_2^{-1} = \Delta + b_X^2 I, \qquad b_X \coloneqq {\textstyle\frac{1}{2}} \|X\| \eqqcolon {\textstyle \frac{1}{2}}  \left(\sum\nolimits_{i=1}^\ell c_i^2\right)^{1/2}. 
\end{equation}
Let us consider the generated semigroup $e^{-t\Delta_\chi }$ on $L^{2}(\mu_\chi)$, which (by Schwartz's kernel theorem) admits an integral kernel $P_t^\chi \in \mathcal{D'}(G\times G)$   
\[
e^{-t \Delta_\chi} f(x) = \int_G P_t^\chi(x,y)f(y)\, \dd \mu_\chi(y) .
\]
Since $\Delta_\chi$ is left-invariant, the heat semigroup also admits a convolution kernel  $p_t^\chi\in \mathcal{D'}(G)$, i.e.\  
\[
e^{-t\Delta_\chi} f(x) = f* p_t^\chi (x)= \int_G f(xy^{-1})p_t^\chi(y)\, \dd \rho(y).
\]
Observe that $P_t^\chi$ and $p_t^\chi$ are related by the equality
\begin{equation}\label{Ptgamma_ptgamma}
P_t^\chi(x,y)= p_t^\chi(y^{-1}x)\chi^{-1}(y)\delta(y).
\end{equation} 
By Hunt's theorem~\cite{Hunt}, $p_t^\chi$ are probability densities. By~\eqref{unitary_equivalence} and the spectral theorem 
\[
e^{-t\Delta_\chi} =  e^{-t b_X^2} \mathcal{U}_2^{-1}  e^{-t\Delta} \mathcal{U}_2 ,
\]
so that
\begin{align*}
P_t^\chi(x,y)
&= e^{-t b_X^2}P_t(x,y) \chi^{-1/2}(x)\chi^{1/2}(y) \nonumber \\
& = e^{-t b_X^2}p_t(y^{-1}x) \chi^{-1/2}(y^{-1}x)\chi^{-1}(y)\delta(y).
\end{align*}
This and~\eqref{Ptgamma_ptgamma} show that
\begin{equation}\label{ptgamma_pt}
p_t^\chi(x) = e^{-t b_X^2} p_t(x) \chi^{-1/2}(x),
\end{equation}
so that both $p_t^\chi$ and $P_t^\chi$ are smooth on $G$ and $G\times G$, respectively. In the following lemma, we provide various estimates of the heat kernel $p_t^\chi$ and of its derivatives that will be of fundamental importance in the following.  

\begin{lemma}\label{estimates_ptgamma} 
The following hold:
\begin{itemize}
\item[(i)] $e^{-t\Delta_\chi}$ is a diffusion semigroup on $(G, \mu_\chi)$;
\smallskip
\item[(ii)] for every $r>1$, $\int_{B_r} \chi \, \dd \rho \leq e^{(\|X\| + D) r}$. 
\end{itemize}
Moreover, there exist two constants $\omega, b>0$ such that, for every $m\in \N$ and $J\in \mathfrak{l}^m$, 
\begin{itemize}
\item[(iii)]  $|X_J p_t^\chi(x)|\lesssim \chi^{-1/2}(x) t^{-(d+m)/2} e^{\omega t} e^{-b|x|^2/t}$, for every $t>0$ and $x\in G$.
\end{itemize}
\end{lemma}

\begin{proof}
The point~(i) is the content of~\cite[Proposition 3.1, (ii)]{HMM}. The point~(ii) is a consequence of~\cite[Proposition~5.7~(ii)]{HMM} and~\eqref{pallegrandi}, which yield
\[
\int_{B_r} \chi \, \dd \rho \leq \left(\sup_{x\in B_r }\chi(x)\right) V(r) \leq e^{\|X\|r}e^{D r}.
\]
Finally,~(iii) follows from~\cite[Section 2.1, eq.~(2.7)]{PV} together with~\eqref{ptgamma_pt} and the observation that if $f$ is a smooth function, then
\[
X_J (\chi  f) =\chi \sum_{I\subseteq J} c(I,\chi) X_I f
\]
for suitable coefficients $c(I,\chi)$, by~\eqref{Xdelta}.
\end{proof}

\section{Sobolev spaces}\label{Sec_Sobolev}
The operator $\Delta_\chi$ on $L^2(\mu_\chi)$ generates a contraction semigroup $e^{-t\Delta_\chi}$ as explained above. This semigroup extends to a bounded (contraction) semigroup on $L^p(\mu_\chi)$ for every $p\in [1,\infty)$ (see e.g.~\cite[Proposition 3.1, (ii)]{HMM}) whose infinitesimal generator, with a slight abuse of notation, we still denote with $\Delta_{\chi}$. For every $p\in (1,\infty)$, the closed operator $\Delta_{\chi}$ satisfies the assumption of Komatsu~\cite[p.\ 286]{Komatsu1} and the fractional powers of $\Delta_{\chi}$ may then be defined following~\cite[Sect.\ 4]{Komatsu1}. Therefore, for $\alpha\geq 0$ we define 
\[
L^p_\alpha(\mu_\chi) \coloneqq \left\{f\in L^p(\mu_\chi) \colon   \Delta_{\chi}^{\alpha/2}f\in L^p(\mu_\chi)\right\} 
\]
endowed with the norm
\[
\|f\|_{L^p_\alpha(\mu_\chi)}\coloneqq \|f\|_{L^p(\mu_\chi)} +  \|\Delta_{\chi}^{\alpha/2}f\|_{L^p(\mu_\chi)} .
\]
Observe that by a result of Komatsu~\cite[Theorem 6.4]{Komatsu1} together with the boundedness of the operator $(\Delta_\chi +cI)^{-\beta}$ on $L^p(\mu_\chi)$ for every $1<p<\infty$, $\beta\geq 0$ and $c>0$, we obtain that 
\begin{equation}\label{equivNorms}
\|f \|_{L^p_\alpha(\mu_\chi)}\approx  \| (\Delta_\chi +c I)^{\alpha/2} f\|_{L^p(\mu_\chi)}.
\end{equation}
and moreover (see~\cite[Section 5]{Komatsu1})
\[
\| (\Delta_\chi +c I)^{\alpha_2 /2} f\|_{L^p(\mu_\chi)} \leq \| (\Delta_\chi +c I)^{\alpha_1 /2}f \|_{L^p(\mu_\chi)} 
\]
whenever $\alpha_1>\alpha_2$, which imply the continuity of the embedding $L^p_{\alpha_1} (\mu_\chi)\hookrightarrow L^p_{\alpha_2}(\mu_\chi)$. 
 
\smallskip

We now state some interpolation properties of the Sobolev spaces. Given a compatible couple of Banach spaces $A$ and $B$, we denote with $[A,B]_\theta$ the intermediate space of index $\theta \in (0,1)$ in the complex method (see~\cite{BerghLofstrom}). 

\begin{lemma}\label{lemma_interpolation}
Let $\theta\in (0,1)$ and $\chi_0$, $\chi_1$ be positive characters of $G$.
\begin{itemize}
\item[(i)] If $0\leq \alpha< \beta$ and $p>1$, then
\begin{equation}\label{intepolation_alphabeta}
[L^p_{\alpha}(\mu_\chi), L^p_{\beta}(\mu_\chi)]_\theta = L^p_{(1-\theta) \alpha + \theta\beta}(\mu_\chi);
\end{equation}
\item[(ii)] if $1\leq p_0,p_1 <\infty$, then
\begin{equation}\label{intepolation_ellr}
[L^{p_0}(\mu_{\chi_0}), L^{p_1}(\mu_{\chi_1})]_\theta = L^{p_\theta}(\mu_{\chi_\theta})
\end{equation}
where
\[
\frac{1}{p_\theta}= \frac{1-\theta}{p_0} + \frac{\theta}{p_1}, \qquad \chi_\theta= \chi_0^{p(1-\theta)/p_0} \chi_1^{p\theta/p_1}. 
\]
\end{itemize}
\end{lemma}

\begin{proof}
We first prove (i). Since for every $c>0$ the operator $\Delta_{\chi} + I$ is positive on $L^p(\mu_\chi)$ in the sense of Triebel~\cite[Chapter 1.14.1, p.\ 91]{Triebel},~\eqref{intepolation_alphabeta} follows from the $L^p(\mu_\chi)$-boundedness ($1<p<\infty$) of the imaginary powers $\Delta_\chi^{iu}$ (see~\cite{Meda}) together with~\cite[Chapter 1.15.3, p.\ 103]{Triebel} and~\eqref{equivNorms}.  
 Case (ii)\ is a special case of \cite[Theorem 5.5.3]{BerghLofstrom}. 
\end{proof}

We provide	now a characterization of the Sobolev spaces in the case when the regularity is an integer (Proposition~\ref{prop:kintero}) and a recursive characterization (Proposition~\ref{alpha+1}), in the same spirit of~\cite{CRTN, PV}. Both results hinge on the $L^p(\mu_\chi)$-boundedness ($1<p<\infty$) of the local Riesz transforms of any order associated with $\Delta_\chi$ which are stated in the following theorem, together with endpoint results at $p=1$ and $p=\infty$ involving the spaces $\mathfrak{h}^1(\mu_\chi)$ and $\mathfrak{bmo}(\mu_\chi)$ introduced in Subsection~\ref{Subsec_HardyBMO}. 
\begin{theorem}\label{teoLocalRiesz}
Let $\chi$ be a positive character of $G$. If $c>0$ is large enough, then for every $m \in \N$ and $J\in \mathfrak{l}^m$ the local Riesz transforms $ X_J (\Delta_\chi +cI)^{-m/2}$ are bounded from $\mathfrak{h}^1(\mu_\chi)$ to $L^1(\mu_\chi)$, from $L^\infty(\mu_\chi)$ to $\mathfrak{bmo}(\mu_\chi)$ and on $L^p(\mu_\chi)$ for every $p\in (1,\infty)$. 
\end{theorem}
The proof of Theorem~\ref{teoLocalRiesz} is postponed to Section~\ref{Sec_Proof_Riesz}. We now proceed to state and prove the aforementioned characterizations of Sobolev norms. 

\begin{proposition}\label{prop:kintero}
Let $k\in \N$ and $p\in (1,\infty)$. Then
\[
\| f\|_{L^p_k(\mu_\chi)} \approx \sum_{ J\in \mathfrak{l}^m, \, m\leq k } \| X_J f\|_{L^p(\mu_\chi)}. 
\]
\end{proposition}
\begin{proof}
By means of~\eqref{equivNorms}, it is enough to prove that for some $c>0$
\begin{equation}\label{cor:step}
\| (\Delta_\chi +c I)^{k/2} f\|_{L^p(\mu_\chi)} \approx \sum_{ J\in \mathfrak{l}^m, \, m\leq k } \| X_J f\|_{L^p(\mu_\chi)}.
\end{equation}
The $\gtrsim $ part of~\eqref{cor:step} amounts to saying 
\[
\| f\|_{L^p(\mu_\chi)} \gtrsim \sum_{ J\in \mathfrak{l}^m, \, m\leq k } \| X_J(\Delta_\chi +c I)^{-k/2} f\|_{L^p(\mu_\chi)},
\]
which holds because for every $J \in \mathfrak{l}^m$ with $m\leq k$,
\[
X_J(\Delta_\chi +c I)^{-k/2} = X_J(\Delta_\chi +c I)^{-m/2} (\Delta_\chi +c I)^{(m-k)/2}
\]
which is bounded on $L^p(\mu_\chi)$ by Theorem~\ref{teoLocalRiesz} and the boundedness of $(\Delta_\chi +c I)^{\beta/2}$ for every $\beta<0$. 

We now prove the inequality $\lesssim$. We first prove the case $k=1$. For every $f\in C_c^\infty(G)$ 
\begin{align*}
\|(\Delta_\chi+cI)^{1/2} f \|_{L^p(\mu_\chi)}
&= \sup \left\{ ((\Delta_\chi+cI)^{\frac12} f,g) \colon\|g\|_{L^{p'}(\mu_\chi) }=1\right\}\\ 
& =\sup \left\{((\Delta_\chi+cI) f,(\Delta_\chi+cI)^{-\frac12}g)\colon \|g\|_{L^{p'}(\mu_\chi)}=1\right\} \\ 
& =\sup \left\{ \sum\nolimits_{j=1}^\ell (X_j^{*,\chi}X_j f + \frac{c}{\ell}f ,(\Delta_\chi+cI)^{-\frac12}g)\colon  \|g\|_{L^{p'}(\mu_\chi)}=1\right\}\\ 
& \lesssim \|f\|_{L^p(\mu_\chi)} + \sup \left\{ \sum\nolimits_{j=1}^\ell (X_j f ,X_j(\Delta_\chi+cI)^{-\frac12}g)\colon \|g\|_{L^{p'}(\mu_\chi)}=1\right\} \\ 
& \lesssim \|f\|_{L^p(\mu_\chi)} + \sum\nolimits_{j=1}^\ell \|X_j f\|_{L^p(\mu_\chi)},
\end{align*}
where $(h,g)$ stands for the duality between $h\in L^p(\mu_\chi)$ and $g\in L^{p'}(\mu_\chi)$. In the last inequality we used the boundedness of the local Riesz transforms on $L^{p'}(\mu_\chi)$ (Theorem~\ref{teoLocalRiesz}). Therefore the inequality $\lesssim$ in~\eqref{cor:step} is proved for $k=1$. 

We can now prove it for every $k$. If $k$ is even, the inequality
\[
\|(\Delta_\chi+cI)^{k/2} f \|_{L^p(\mu_\chi)} \lesssim \sum_{ J\in \mathfrak{l}^m, \, m\leq k } \|X_J f\|_{L^p(\mu_\chi)} 
\]
is straightforward. If $k\geq 3$ is odd, then
\begin{align*}
\|(\Delta_\chi+cI)^{k/2} f \|_{L^p(\mu_\chi)} 
&= \|(\Delta_\chi+cI)^{1/2}(\Delta_\chi+cI)^{(k-1)/2} f \|_{L^p(\mu_\chi)} \\
& \lesssim \sum_{i=1}^\ell \|X_i(\Delta_\chi+cI)^{(k-1)/2} f \|_{L^p(\mu_\chi)} \lesssim \sum_{ J\in \mathfrak{l}^m, \, m\leq k } \|X_J f\|_{L^p(\mu_\chi)}
\end{align*}
since the part $\lesssim$ of~\eqref{cor:step} holds for $k=1$. The proof is complete.
\end{proof}

\begin{proposition}\label{alpha+1}
For every $\alpha\geq 0$ and $p\in (1,\infty)$, $f\in L^p_{\alpha+1}(\mu_\chi)$ if and only if $f\in L^p_\alpha(\mu_\chi)$ and $X_i f\in L^p_\alpha(\mu_\chi)$ for every $i\in \mathfrak{l}$. In particular
\[
\|f\|_{L^p_{\alpha+1}(\mu_\chi)} \approx \|f\|_{L^p_{\alpha}(\mu_\chi)} +\sum_{i=1}^\ell \|X_i f\|_{L^p_{\alpha}(\mu_\chi)}.
\]
\end{proposition}
\begin{proof}
We claim that the operators $R_i \coloneqq (\Delta_\chi+ I)^{-1/2}X_i$ are bounded on $L^p_\beta(\mu_\chi)$ for every $\beta\geq 0$ and $i\in \mathfrak{l}$. Assuming the claim, we prove the statement. 

We begin by proving the inequality $\gtrsim$. By~\eqref{equivNorms}, the continuity of the embedding $L^p_{\alpha+1}(\mu_\chi)\hookrightarrow L^p_\alpha(\mu_\chi)$ and the claim
\begin{align*}
\|f\|_{L^p_\alpha(\mu_\chi)} + \sum_{i=1}^\ell \|X_i f\|_{L^p_\alpha(\mu_\chi)} 
&\approx \|f\|_{L^p_\alpha(\mu_\chi)}  + \sum_{i=1}^\ell  \|(\Delta_\chi + I)^{(\alpha+1)/2} R_i f\|_{L^p(\mu_\chi)} \\
& \lesssim  \|f\|_{L^p_{\alpha+1}(\mu_\chi)} + \sum_{i=1}^\ell  \|R_i f\|_{L^p_{\alpha+1}(\mu_\chi)} \lesssim  \|f\|_{L^p_{\alpha+1}(\mu_\chi)}.
\end{align*}
The converse inequality may be obtained in a similar way, since by~\eqref{equivNorms} and the claim
\begin{align*}
\|f\|_{L^p_{\alpha+1}(\mu_\chi)}
&\approx \|(\Delta_\chi + I)^{-1/2}(\Delta_\chi + I) f\|_{L^p_\alpha(\mu_\chi)}  \\
& \lesssim \sum_{i=1}^\ell  \|(\Delta_\chi + I)^{-1/2} (X_i^2  - c_i X_i+1) f\|_{L^p_\alpha(\mu_\chi)} \lesssim  \|f\|_{L^p_{\alpha}(\mu_\chi)} + \sum_{i=1}^\ell  \|X_i f\|_{L^p_{\alpha}(\mu_\chi)}.
\end{align*}
It remains to prove the claim. By the interpolation properties~\eqref{intepolation_alphabeta} of the Sobolev spaces, it is enough to prove it when $\alpha=k$ is an integer. But this holds since by Proposition~\ref{prop:kintero}
\begin{align*}
\| (\Delta_\chi + I)^{-1/2} X_i f\|_{L^p_{k}(\mu_\chi)}
& \approx  \|(\Delta_\chi + I)^{(k-1)/2} X_i f\|_{L^p(\mu_\chi)} \\
& \lesssim \sum_{ J\in \mathfrak{l}^m, \, m\leq k } \|X_J f\|_{L^p(\mu_\chi)} \approx \| f\|_{L^p_k(\mu_\chi)},
\end{align*}
and the claim is proved.
\end{proof}
In the following proposition we provide an alternative characterization of the Sobolev norm when $p\in (1,\infty)$. Observe that a similar result does not hold when $p=\infty$.

\begin{proposition}\label{deltadentro}
Let $p\in (1,\infty)$ and $\alpha\geq 0$. Then the operator $\mathcal{U}_p \colon f \mapsto \chi^{1/p} f$ is a Banach space isomorphism between $L^p_\alpha(\mu_\chi)$ and $L^p_\alpha(\rho)$. In particular
\[
\| f\|_{L^p_\alpha(\mu_\chi)} \approx \| \chi^{1/p} f\|_{L^p_\alpha(\rho)}.
\]
\end{proposition}
\begin{proof}
By~\eqref{Xdelta}, it is easy to see that for every $k \in\N$ and every $p\in (1,\infty)$ 
\begin{equation*}
\sum_{J\in \mathfrak{l}^m, m \leq k }\|\chi X_J f\|_{L^p(\rho)} \approx \sum_{J\in \mathfrak{l}^m, m \leq k }\|X_J  (\chi f)\|_{L^p(\rho)}.
\end{equation*}
This and Proposition~\ref{prop:kintero} prove the statement when $\alpha$ is an integer. Otherwise, let $\alpha = k+\theta$ where $k\in \N$ and $\theta \in (0,1)$. Then
\begin{align*}
\mathcal{U}_p^{-1} L^p_{\alpha}(\rho)
&=\mathcal{U}_p^{-1} [L^p_k(\rho), L^p_{k+1}(\rho)]_\theta\\ 
&= [\mathcal{U}_p^{-1} L^p_k(\rho), \mathcal{U}_p^{-1} L^p_{k+1}(\rho)]_{\theta}\\
&= [L^p_{k}(\mu_\chi), L^p_{k+1}(\mu_\chi)]_{\theta} \\
&= L^p_{\alpha}(\mu_\chi),
\end{align*}
the first and last equalities by~\eqref{intepolation_alphabeta} and the second equality by definition of interpolation in the complex method. Since all equalities are with equivalence of norms, the statement is proved.
\end{proof}

\section{Sobolev embeddings}\label{Sec_SET}
Assume that an embedding of the form $L^p_k(\mu_\chi)\hookrightarrow L^q(\mu_\chi)$ holds for some $p\neq q$, $p\in (1,\infty)$, $q\in [1,\infty]$ and $k\in \N$.  Then by Proposition~\ref{prop:kintero} it is equivalent to the requirement
\[
\| f\|_{L^q(\mu_\chi)}\lesssim \sum_{ J\in \mathfrak{l}^m, \, m\leq k } \| X_J f\|_{L^p(\mu_\chi)}.
\]
Assume this holds true, and for $y\in G$ consider the function $L_y f(x) =f(y^{-1}x)$. Then 
\begin{align*}
(\chi\delta^{-1})^{1/q}(y) \|f\|_{L^q(\mu_\chi)}&= \|L_y f\|_{L^q(\mu_\chi)} \\
& \lesssim \sum_{ J\in \mathfrak{l}^m, \, m\leq k } \| X_J(L_y f)\|_{L^p(\mu_\chi)} \\
& = \sum_{ J\in \mathfrak{l}^m, \, m\leq k } \| L_y (X_J f)\|_{L^p(\mu_\chi)} =(\chi\delta^{-1})^{1/p}(y) \sum_{ J\in \mathfrak{l}^m, \, m\leq k } \| X_J f\|_{L^p(\mu_\chi)}
\end{align*}
by the left invariance of the vector fields. This implies
\[
(\chi\delta^{-1})^{1/q-1/p}(y) \|f\|_{L^q(\mu_\chi)}\lesssim \sum_{ J\in \mathfrak{l}^m, \, m\leq k } \| X_J f\|_{L^p(\mu_\chi)}
\]
for every $y\in G$. Since $\chi\delta^{-1}$ is a character of $G$, and characters grow exponentially at infinity (see~\cite[Proposition 5.6, (ii)]{HMM}), then necessarily $\chi=\delta$. In other words, an embedding of the form $L^p_k(\mu_\chi) \hookrightarrow L^q(\mu_\chi)$ with $p\neq q$ may hold only if the measure is the left Haar measure $\lambda$. It is a little more tricky to show that $L^p_k(\mu_{\chi})$ does not even embed, in general, in $\mathfrak{bmo}(\mu_\chi)$ for any $k\in \N$ if $\mu_\chi \neq \lambda$ (see Section~\ref{Sec_ax+b} below). In the following subsection we prove instead that certain Sobolev embeddings for the left Haar measure hold true, as stated in Theorem~\ref{SETright} in the case when $\chi =\delta$.

\smallskip

We shall first prove Theorem~\ref{SETright} in the case when $\chi=\delta$ (see Theorem~\ref{SET} below). From now on, we let $G_{\alpha,\chi}^c$ be the convolution kernel of the operator $(\Delta_\chi + c I)^{-\alpha/2}$, i.e.\
\[
(\Delta_\chi + c I)^{-\alpha/2} f = f* G_{\alpha,\chi}^c, \qquad G_{\alpha,\chi}^c(x) = C(\alpha) \int_0^\infty t^{\alpha/2-1} e^{-ct} p_t^\chi(x) \, \dd t.
\]
\begin{lemma}\label{estimates_Gsigma}
Let $c>\omega$. Then
\[
G_{\alpha,\chi}^c(x)\lesssim 
\begin{cases}
|x|^{\alpha -d} & \mbox{if $0<\alpha <d$}\\ 
\log(1/|x|) & \mbox{if $\alpha =d$}\\
1  & \mbox{if $\alpha >d$},
\end{cases} \qquad \mbox{when } \quad |x|\leq 1,
\]
while
\[
G_{\alpha,\chi}^c(x)\lesssim \chi^{-1/2}(x) e^{-c' |x|} \qquad \mbox{when } \quad |x|> 1,
\]
where $c'= \frac{1}{2}\sqrt{b(c-\omega)}$ and $b$ and $\omega$ are as in Lemma~\ref{estimates_ptgamma}.
\end{lemma}
\begin{proof}
We begin by applying the estimate~(iii) of Lemma~\ref{estimates_ptgamma} with $m=0$. We obtain 
\[
G_{\alpha,\chi}^c(x)\lesssim \chi^{-1/2}(x) \int_0^\infty t^{(\alpha-d)/2 -1}e^{-a t } e^{-b|x|^2/t} \, \dd t, 
\]
where $a =c-\omega>0$. 

If $|x|>1$, then 
\[
-at -b \frac{|x|^2}{t} \leq -\frac{at}{2} -\frac{b}{2t} - \frac{\sqrt{ab}}{2}|x|
\]
so that
\[
\int_0^\infty t^{(\alpha-d)/2 -1}e^{-a t } e^{-b|x|^2/t} \, \dd t \leq e^{-\frac{1}{2}\sqrt{ab}|x|} \int_0^\infty t^{(\alpha-d)/2 -1} e^{-\frac{at}{2}-\frac{b}{2t}}\, \dd t \lesssim e^{-\frac{1}{2}\sqrt{ab}|x|} .
\]
This yields the desired estimate with $c' = \frac{1}{2}\sqrt{ab}$.

Let $|x|\leq 1$ and observe that $\chi^{-1/2}(x)\approx 1$. We split the integral into the integrals over $(0,1)$ and $(1,\infty)$, which yields 
\[
G_{\alpha,\chi}^c(x)\lesssim \int_0^1 t^{(\alpha-d)/2 -1} e^{-b|x|^2/t} \,
\dd t  +  \int_1^\infty  e^{-a t }e^{-b|x|^2/t} \, \dd t \eqqcolon
G_1(x) + G_2(x), 
\]
where $G_2(x)\lesssim 1$. If $\alpha >d$, then also $G_1(x)\lesssim 1$ since $e^{-b|x|^2/t} \leq 1$. Otherwise, we make the change of variables $|x|^2/t =u$ and obtain 
\[
G_1(x) = |x|^{\alpha -d} \int_{|x|^2}^\infty u^{(d-\alpha)/2} e^{-bu} \, \frac{\dd u}{u}
\]
from which the first estimate follows.
\end{proof}

\begin{corollary}\label{cor:deltaG}
Let $a,s\in \R$, $r>1$. If $c>0$ is sufficiently large, then
\[
\|\delta^a \chi^{s} G_{\alpha,\chi}^c\|_{L^r(\rho)}<\infty \qquad \forall\, \alpha>d \frac{r-1}{r}.
\]
\end{corollary}
\begin{proof}
In view of Lemma~\ref{estimates_Gsigma} we split
\begin{equation}\label{GsigmaLrnorm}
\| \delta^a\chi^s G_{\alpha,\chi}^c \|_{L^r(\rho)}^r = \int_{B_1} |\delta^a \chi^{s}  G_{\alpha,\chi}^c|^r \, \dd \rho +  \int_{B_1^c} |\delta^a \chi^{s}  G_{\alpha,\chi}^c|^r \, \dd \rho. 
\end{equation}
The integral at infinity can be treated by Lemma~\ref{estimates_Gsigma} together with Lemma~\ref{estimates_ptgamma}~(ii). These yield 
\begin{align*}
 \int_{B_1^c} |\delta^a \chi^{s} G_{\alpha,\chi}^c|^r \, \dd \rho 
 & \leq \int_{B_1^c} \delta^{ra}(x) \chi^{rs}(x)  e^{-rc'|x|} \, \dd \rho(x) \\
 & = \sum_{k=0}^\infty  e^{-rc'2^k} \int_{A^1_k} \delta^{ra}(x) \chi^{rs}(x)  \, \dd \rho(x)  \lesssim \sum_{k=0}^\infty  e^{-rc'2^k} e^{C2^k}<\infty 
\end{align*}
if $c$ (hence $c'$) is large enough. 

To treat the local integral, we first observe that both $\chi$ and $\delta$ are bounded on $B_1$ and play no role. If $\alpha <d$ then 
\begin{align*}
\int_{B_1} | G_{\alpha,\chi}^c|^r \, \dd \rho 
&\lesssim \int_{B_1} |x|^{r(\alpha-d)} \, \dd \rho \lesssim \int_0^1 u^{r(\alpha-d)} u^{d-1}\, \dd u <\infty 
\end{align*}
whenever
\[
\alpha > d\frac{r-1 }{r}.
\]
The case $\alpha \geq d$ can be treated again by means of Lemma~\ref{estimates_Gsigma}, but it is even easier and omitted. 
\end{proof}
\begin{lemma}[Young's inequalities]\label{Young}
For $1<p\leq q<\infty$, choose $r\geq 1$ such that $\frac{1}{p}+ \frac{1}{r}=1+\frac{1}{q}$. Then
\begin{equation}\label{Youngfinite}
\|f*g\|_{L^q(\lambda)} \leq \|f\|_{L^p(\lambda)}\left( \| \check{g}\|_{L^r(\lambda)}^{r/p'} \|g\|_{L^r(\lambda)}^{r/q}\right)
\end{equation}
where $\check{g}(x)= g(x^{-1})$. If $p>1$ and $r\geq 1$ is such that $\frac{1}{p}+ \frac{1}{r}=1$, then
\begin{equation}\label{Younginfinite}
\|f*g\|_\infty \leq \|f\|_{L^p(\lambda)} \|\check{g}\|_{L^{p'}(\lambda)}.
\end{equation}
\end{lemma}
\begin{proof}
The inequality~\eqref{Youngfinite} can be found in the proof of~\cite[(20.18)]{HewittRoss}, while~\eqref{Younginfinite} is~\cite[(20.16)]{HewittRoss}. 
\end{proof}

We are now ready to prove the Sobolev embedding theorem for $L^p_\alpha(\lambda)$.
\begin{theorem}\label{SET}
Let $1< p,q <\infty$.
\begin{itemize}
\item[(a)] If $\alpha>0$, then $L^p_\alpha (\lambda) \hookrightarrow L^q(\lambda)$ for every $q\geq p$ such that $\frac1p-\frac1q \leq \frac\alpha d$; 
\item[(b)] if $\alpha\geq d/p$, then $L^p_\alpha (\lambda) \hookrightarrow L^q(\lambda)$ for every $q\geq  p$;
\item[(c)] if $\alpha> d/p$, then $L^p_\alpha(\lambda) \hookrightarrow L^{\infty}$.
\end{itemize}
\end{theorem}

\begin{proof}

We begin with (a), and observe that it is enough to prove it when $1<p\leq q <\infty$, $\alpha>0$ and $\frac1p-\frac1q =\frac\alpha d$, because the intermediate terms follow from this by interpolation. Indeed, assume that
\begin{equation}\label{SETaEqual}
L^p_\alpha(\lambda) \hookrightarrow L^q(\lambda), \qquad \frac1p-\frac1q =\frac\alpha d.
\end{equation}
Since $L^p_\alpha(\lambda) \hookrightarrow L^p(\lambda)$ by Lemma~\ref{lemma_interpolation}, we obtain that for every $\theta\in (0,1)$
\[
L^p_\alpha(\lambda)= [L^p_\alpha(\lambda), L^p_\alpha(\lambda)]_\theta \hookrightarrow [L^p(\lambda), L^q(\lambda)]_\theta = L^{q_\theta} (\lambda), \qquad \frac{1}{q_\theta}= \frac{1-\theta}{p} + \frac{\theta}{q},
\]
where
\[
q_\theta \geq p, \quad \frac{1}{p} - \frac{1}{q_\theta} = \theta \frac{\alpha}{d}<\frac{\alpha}{d}.
\]
Assume now $p\geq q$ is such that $\frac1p-\frac1q = \frac 1 d$. By \cite[\textsection 6, Theorem 1]{Varopoulos1}
\begin{equation}\label{embed:Var}
\|  f\|_{L^q(\lambda)} \lesssim \| f\|_{L^p_1(\lambda)},
\end{equation}
which is~\eqref{SETaEqual} when $\alpha =1$.

Let now $\alpha\in (0,1)$. By~\eqref{embed:Var} and Lemma~\ref{lemma_interpolation}, if $q\geq p$ and $\frac{1}{p}-\frac{1}{q}=\frac{1}{d}$, then
\[
 L^p_\alpha(\lambda) = \left[ L^p(\lambda), L^p_1(\lambda) \right]_\alpha \hookrightarrow \left[ L^p(\lambda), L^q(\lambda)\right]_\alpha = L^r (\lambda) , \qquad \frac 1r=\frac{1-\alpha}{p}+\frac{\alpha}{q},
\]
which is the statement for $\alpha \in (0,1)$, since $r\geq p$ and $\frac{1}{p} - \frac{1}{r} =  \frac{\alpha}{d}$. Thus, the statement is proved for every $\alpha \in (0,1]$.

We finally prove that if the statement holds for some $\alpha$, then it holds also for $\alpha +1$. Let $q\geq p$ be such that
\[
\frac1p-\frac1q\le \frac{\alpha+1}{d},
\]
and take $r$ such that $p<r<q$ and
that
\[
\frac1p-\frac1r\le \frac{\alpha}{d}\quad\text{and}\quad 
\frac1r-\frac1q\le \frac1d \,.
\]
This choice implies, by inductive assumption, that $L^p_\alpha(\lambda) \hookrightarrow L^r(\lambda)$ and $L^r_1(\lambda) \hookrightarrow L^q(\lambda)$. Then
\[
\begin{aligned}
\| f\|_{L^q (\lambda) }
&\lesssim \| f\|_{L^r_1(\lambda)}  \\
& \lesssim \| f\|_{L^r(\lambda)} + \sum_{i=1}^\ell \| X_i f\|_{L^r(\lambda)} \\
& \lesssim \|f\|_{L^p_\alpha(\lambda)} + \sum_{i=1}^\ell \| X_i f\|_{L^p_\alpha(\lambda)} \lesssim \| f\|_{L^p_{\alpha+1}(\lambda)},
\end{aligned}
\]
the second and fourth inequalities by Proposition~\ref{alpha+1}. The proof of (a) is complete.

\smallskip

To prove (b), we shall prove that if $1< p\leq q <\infty$ and $\alpha \geq d/p$, then $(\mathcal{L}+ cI )^{-\alpha/2} \colon L^p(\lambda) \to
L^q(\lambda)$ if $c>0$ is large enough.

To do this, choose $r\geq 1$ such that $\frac{1}{p}+ \frac{1}{r}=1+\frac{1}{q}$. Then, applying~\eqref{Youngfinite} 
\begin{align*}
\| (\mathcal{L} + cI )^{-\alpha/2} f\|_{L^q(\lambda)} & 
= \| f *G_{\alpha,\delta}^c\|_{L^q(\lambda)}\\
&\lesssim \| f\|_{L^p(\lambda) }\|
\check{G}_{\alpha,\delta}^c\|_{L^r(\lambda)}^{1-r/q}
\|G_{\alpha,\delta}^c\|_{L^r(\lambda)}^{r/q}. 
\end{align*}
Since $d\frac{r -1 }{r} = \frac{d}{p} - \frac{1}{q}$, (b) follows by Corollary~\ref{cor:deltaG}. 

\smallskip

We finally prove (c). Let $1< p,q <\infty$ and $\alpha> d/p$. We prove that $(\mathcal{L}+ cI )^{-\alpha/2} \colon L^p(\lambda) \to
L^{\infty}$ if $c>0$ is large enough. By~\eqref{Younginfinite}, 
\[
\| f *G_{\alpha,\delta}^c\|_\infty = \| f*
G_{\alpha,\delta}^c\|_\infty  \lesssim \|f\|_{L^p(\lambda)} \|
\check{G}_{\alpha,\delta}^c\|_{L^{p'}(\lambda)}. 
\]
It then remains to prove that $\check{G}_{\alpha,\delta}^c \in L^{p'}(\lambda)$ when $\alpha> d/p$ and $c>0$ is large enough. But
this follows by Corollary~\ref{cor:deltaG}, since $\frac{p'-1}{p'}=\frac{1}{p}$. 
\end{proof}

We finally prove Theorem~\ref{SETright}. We recall that if $\kappa$ is a positive character of the group, we denote with $\kappa L^{\infty}$ the Banach space of functions $\{ \kappa f\colon   f\in L^\infty \}$ endowed with the norm $\| \kappa f \|_{ \kappa L^\infty } \coloneqq \| f\|_{\infty}$.

\begin{proof}[Proof of Theorem~\ref{SETright}]
The statements (a) and (b) follow immediately from the statements (a) and (b) of Theorem~\ref{SET} by Proposition~\ref{deltadentro}. The proof of statement (c) is an easy modification of the proof of Theorem~\ref{SET}, (c). Indeed, observe that by Lemma~\ref{Young},~\eqref{Younginfinite}, one can deduce the inequality
\begin{align*}
\| (\chi\delta^{-1})^{1/p} (f *G_{\alpha,\chi}^c)\|_\infty 
&= \| [(\chi\delta^{-1})^{1/p}f] *[(\chi\delta^{-1})^{1/p} G_{\alpha,\chi}^c]\|_\infty \\
& \lesssim \|f\|_{L^p(\mu_\chi)} \| (\chi^{-1} \delta)^{1/p} \check{G}_{\alpha,X}^c\|_{L^{p'}(\lambda)}.
\end{align*}
This proves the statement by Corollary~\ref{cor:deltaG}.
\end{proof}
\section{Algebra properties}\label{Sec_algebra}
For every $\alpha \in (0,1)$ and $R>0$ we define the operators
\[
S^{R}_{\alpha}f(x)\coloneqq \left( \int_0^R \left[   \frac{1}{u^{\alpha}V(u)}\int_{|y|<u}|f(xy^{-1})-f(x)|\, \dd \rho(y) \right]^2\frac{\dd u}{u}  \right)^{1/2}
\]
and $S^\loc_\alpha \coloneqq S^1_\alpha$. The following theorem is a key step to prove Theorem~\ref{algebraproduct}.
\begin{theorem}\label{Salphaloc} 
Let $\chi$ be a positive character of $G$. For every $p\in (1,\infty)$ and $\alpha\in (0,1)$,
\[
 \|f\|_{L^{p}_\alpha(\mu_\chi)} \approx   \|S^{{\loc}}_{\alpha}f\|_{L^p(\mu_\chi)}+\|f\|_{L^p(\mu_\chi)}.
\]
 \end{theorem}
 \begin{proof}
By~\cite[Theorem 1.3 (i)]{PV}, Proposition~\ref{deltadentro} and the definition of $S^\loc_\alpha$, it is enough to prove that
\begin{equation}\label{toproof}
\chi^{1/p}|f| + S^R_{\alpha} (\chi^{1/p}f) \approx \chi^{1/p} |f| +  \chi^{1/p} S_\alpha^R f
\end{equation}
when $R=1$. For future convenience and since the proof is the same, we consider $R>0$. 

To prove the inequality $\lesssim$, we take $x\in G$, $u\in (0,R)$ and $y\in B_u$. By the triangle inequality
\begin{align*}
| \chi^{1/p}&(xy^{-1}) f(xy^{-1}) - \chi^{1/p}(x)f(x)| \\
& \leq \chi^{1/p}(x)  \left( \chi^{-1/p}(y)|f(xy^{-1}) -f(x)| + | \chi^{1/p}(y^{-1}) -1||f(x)|\right).
\end{align*}
Since $y\in B_u\subseteq B_R$, $\chi^{-1/p}(y)\approx 1$. Moreover, by the smoothness of $\chi$ and~\eqref{Xdelta}, its derivative at the identity does not vanish and then $| \chi^{1/p}(y^{-1}) -1|\approx |y|\leq u$. Therefore
\begin{align*}
S^R_{\alpha} (\chi^{1/p}f)(x)\lesssim \chi^{1/p}(x) (S_\alpha^R f)(x) + \chi^{1/p}(x) |f(x)|.
\end{align*}
The converse inequality is analogous, since
\begin{align*}
\chi^{1/p}&(x)|f(xy^{-1}) - f(x)| \\
& \lesssim |\chi^{1/p}(xy^{-1}) f(xy^{-1}) - \chi^{1/p}(x)f(x)| + \chi^{1/p}(x) |f(x)| |1 - \chi^{1/p}(y)|.
\end{align*}
This completes the proof.
\end{proof}
We shall also need the following
\begin{proposition}\label{mean}
Let $\varepsilon,\beta, \alpha\geq 0$, $1<p,r<\infty$, $1<q< \infty$ and $0<\theta<1$ be such that $\alpha=\theta \varepsilon+(1-\theta)\beta$ and $1/r=\theta/p+(1-\theta)/q$. Then for all $f\in L^p_{\varepsilon}(\mu_\chi)\cap L^q_{\beta}(\mu_\chi)$ 
\begin{equation}\label{Holder}
\|f\|_{L^r_\alpha(\mu_\chi)}\lesssim \|f\|_{L^p_{\varepsilon}(\mu_\chi)}^{\theta}\|f\|_{L^q_{\beta}(\mu_\chi)}^{1-\theta}.
\end{equation}
If $q=\infty$ and $\beta=0$, so that $\alpha=\theta \varepsilon$ and $1/r=\theta/p$, then for all $f\in L^p_{\varepsilon}(\mu_\chi)\cap L^\infty$ 
\begin{equation}
\|f\|_{L^r_\alpha(\mu_\chi)}\lesssim \|f\|_{L^p_{\varepsilon}(\mu_\chi)}^{\theta}\|f\|_{\infty}^{1-\theta}.
\end{equation}
 \end{proposition}
\begin{proof}
Recall that 
by Lemma~\ref{estimates_ptgamma}~(i), $\Delta_\chi$ generates a
diffusion semigroup and thus Littlewood-Paley-Stein theory applies
together with~\cite{Meda}.  Arguing as in~\cite[Proposition 4.2]{PV},
while considering weighted norms
 and operators, the conclusion follows.  We omit the details.
\end{proof}

\subsection{Proof of Theorem~\ref{algebraproduct}}
This proof cannot be derived directly
from~\cite[Theorem 1.2]{PV}, though its proof follows the same lines. 
Hence, we provide all its necessary ingredients.

Let $R>0$, $\alpha \in (0,1)$ and $q\in [1,\infty)$. We define the maximal operator
\begin{equation}\label{MRbeta}
M^R f(x)\coloneqq \sup_{B\in\mathcal B_R,\, x\in B} \frac{1}{\rho(B)} \int_B|f|	\, \dd \rho.
\end{equation}
We recall that $M^{R}$ is bounded on $L^p(\rho)$ for every $p\in (1,\infty]$ and it is of weak type $(1,1)$. By~\eqref{localdoubmugamma}, it is easy to see that for every character $\chi$ of $G$ and $R>0$
\begin{equation}\label{MRdeltafuori}
M^R (\chi f) \approx \chi M^R f
\end{equation}
with a constant independent of $f$. Therefore, $M^{R}$ is also bounded on $L^p(\mu_\chi)$ for every $p \in (1,\infty]$ by Proposition~\ref{deltadentro}. 

We also define the functionals
\begin{align*}
\Omega^{(q)}_f(x,r)&\coloneqq\sup\left\{\left(\frac{1}{\rho(B)}\int_B|f-f_B|^q\dd \rho\right)^{1/q}:\, B\in\mathcal B_r,\, x\in B  \right\},\qquad \Omega_f \coloneqq \Omega_f^{(1)},\\
\Omega^{\infty}_f(x,r)&\coloneqq \sup\{ \|f-f_B\|_{\infty} :\, B\in\mathcal B_r,\,  x\in B  \}
\end{align*}
and 
\[
\begin{aligned}
G^{R}_{\alpha,q}f(x)&=\left( \int_0^R[r^{-\alpha} \Omega^{(q)}_f(x,r)]^2   \frac{\dd r}{r}\right)^{1/2}, \quad G^R_{\alpha}\coloneqq G^{R}_{\alpha,1}, \quad G^\loc_{\alpha, q}\coloneqq G^{1}_{\alpha,q}, \quad G^\loc_{\alpha}\coloneqq G^{1}_{\alpha,1}.
\end{aligned}
\]
\begin{lemma}\label{lemmaGlocq}
Let $f$ be a locally integrable function and $\alpha \in (0,1)$. For every $r\in (1,\infty)$
\begin{equation}\label{GlocSloc}
\| G_\alpha^\loc f\|_{L^p(\mu_\chi)} \lesssim \| S_\alpha^\loc f\|_{L^p(\mu_\chi)} + \|f\|_{L^p(\mu_\chi)},
\end{equation}
while
\begin{itemize}
\item[(i)] if $p\in (1,2]$, $\alpha\leq\frac{d}{p}$, $q<\frac{dp}{d-\alpha p}$, then
\[
\|G^{R_1}_{\alpha,q}f\|_{L^p(\mu_\chi)}+\|f\|_{L^p(\mu_\chi)} \approx \|G^{R_2}_{\alpha}f\|_{L^p(\mu_\chi)}+\|f\|_{L^p(\mu_\chi)};
\]
\item[(ii)] if $p\in [2,\infty)$, $\alpha\leq\frac{d}{p}$, $q<\frac{2d}{d-2\alpha}$, then
\[
\|G^{R_1}_{\alpha,q}f\|_{L^p(\mu_\chi)}+\|f\|_{L^p(\mu_\chi)} \approx \|G^{R_2}_{\alpha}f\|_{L^p(\mu_\chi)}+\|f\|_{L^p(\mu_\chi)}.
\]
\end{itemize}
\end{lemma}
\begin{proof}
To prove~\eqref{GlocSloc}, we observe that it is proved in~\cite[pp.\ 27-28]{PV} that 
\begin{equation}\label{GalphaS2}
G_\alpha^\loc f \lesssim S^2_\alpha f.
\end{equation}
By~\cite[Lemma 4.4]{PV}, we have that $\|S_\alpha^{R_1}f\|_{L^p} + \|f\|_{L^p} \approx \|S_\alpha^{R_2}f\|_{L^p} +  \|f\|_{L^p}$ for every $R_1,R_2>0$. By Proposition~\ref{deltadentro} and~\eqref{toproof} then
\[
\|S_\alpha^{R_1}f\|_{L^p(\mu_\chi)} + \|f\|_{L^p(\mu_\chi)} \approx \|S_\alpha^{R_2}f\|_{L^p(\mu_\chi)} +  \|f\|_{L^p(\mu_\chi)}.
\]
This together with~\eqref{GalphaS2} yield~\eqref{GlocSloc}.

We now prove (i). The inequality $\gtrsim$ is straightforward and holds for every $p\in (1,\infty)$ and $q>1$. To prove the converse inequality, we take $p\in (1,2]$, $\alpha\leq \frac{d}{p}$, $q<\frac{dp}{d-p\alpha}$ and we choose $p_0<p$ and $0\leq \beta <\frac{\alpha}{d}\leq \frac{1}{p}<1$ such that $\frac{1}{q}=\frac{1}{p_0}-\beta$. The starting point is the inequality
\begin{equation}\label{ineqPVpreHardy}
\Omega_f^{(q)}(x,r)\lesssim V({r})^{\beta} \int_0^{8r}  V(s)^{-\beta}  M^3(|\Omega_f(\cdot,s)|^{p_0})^{\frac{1}{p_0}}(x)\frac{\dd s}{s}
\end{equation}
that can be found in~\cite[p.\ 26]{PV}. We now apply the following version of Hardy's inequality:
\[
\int_0^R \frac{1}{r^\nu} \left( \frac{1}{r} \int_0^r g(s)\, \dd s\right)^p \, \dd r \lesssim \int_0^R \frac{1}{r^\nu} g(r)^p\, \dd r, \qquad p>1,	\quad 1-p<\nu<1,\quad g\geq 0,
\]
and obtain (the first inequality by~\eqref{ineqPVpreHardy}; while Hardy's inequality may be applied since $\alpha>\beta d$)
\begin{align*}
G^R_{\alpha, q}f(x) 
&\lesssim \left( \int_0^R \left[ r^{-\alpha} r^{d\beta} \int_0^{8r} s^{-d\beta} M^3(|\Omega_f(\cdot,s)|^{p_0})^{\frac{1}{p_0}}(x)\frac{\dd s}{s}\right]^2 \frac{dr}{r}\right)^{1/2} \\
& \lesssim \left( \int_0^{8R} r^{-2\alpha-1} \left[ M^3(|\Omega_f(\cdot,r)|^{p_0})^{\frac{1}{p_0}}(x)\right]^2 \, \dd r \right)^{1/2}.
\end{align*}
One can now argue as in~\cite[p.\ 318]{CRTN}, since
\[
\left\| \left[ \sum_{k=0}^\infty (M^3 f_k)^q\right]^{1/q}\right\|_{L^p(\mu_\chi)} \lesssim \left\| \left[ \sum_{k=0}^\infty |f_k|^q\right]^{1/q}\right\|_{L^p(\mu_\chi)}
\]
by~\cite[Proposition 7]{CRTN} together with~\eqref{MRdeltafuori} and Proposition~\ref{deltadentro}. The proof of (ii) is analogous.
\end{proof}

By means of the results stated above, Theorem~\ref{algebraproduct} may be proved exactly as~\cite[Theorem 1.2]{PV}. We omit the details. We finally state the following corollary, which is a consequence of the algebra property of $L^p_\alpha(\lambda) \cap L^\infty$ (Theorem~\ref{algebraproduct}) together with the embedding $L^p_\alpha(\lambda)\hookrightarrow L^\infty$ (Theorem~\ref{SET}, (c)).
\begin{corollary}
For every $p\in (1,\infty)$ and $\alpha>d/p$, $L^p_\alpha(\lambda)$ is an algebra.
\end{corollary}

\subsection{Negative results: the affine group of the line}\label{Sec_ax+b} In this subsection we show that an algebra property for the Sobolev spaces $L^p_\alpha(\mu_\chi)$, with $\mu_\chi \neq \lambda$, cannot hold in general on nonunimodular Lie groups, for it fails on the ``$ax+b$'' group.
 
We recall that the affine group of the line, usually referred to as the ``$ax+b$'' group, is the topological space $G = \R \ltimes \R_+$ endowed with the product law $(x,a)(x',a')=(x+ax',aa')$. The identity is $e= (0,1)$, the right Haar measure is $\dd \rho(x,a) = a^{-1} \, \dd x\, \dd a$ and the modular function is $\delta(x,a)=a^{-1}$. A basis for the Lie algebra $\mathfrak{g}$ is $\{X_0,X_1\}$, where $X_0 = a\partial_a$ and $X_1= a\partial_x$.  

\smallskip

Let $G$ be the ``$ax+b$'' group. If $\chi$ is a positive character of $G$, then its differential $\chi'$ vanishes on $[\mathfrak{g},\mathfrak{g}]=\langle X_0\rangle$. Then $\partial_x \chi=0$, i.e.\ $\chi(x,a)=\chi(0,a)$. Since $\chi (0,a)$ is a positive character of $\R^+$, the character $\chi$ must be of the form $\chi(x,a)= a^{-\gamma}$ for some $\gamma \in \R$. In other words, all characters of $G$ are of the form $\delta^\gamma$ for some $\gamma \in \R$.  

In the following proposition, we show that if $\gamma\neq 1$ (i.e.\ $\mu_{\delta^\gamma}\neq \lambda$), then $L^p_\alpha(\mu_{\delta^\gamma})$ is never an algebra; and that in this case the embedding of $L^p_\alpha(\mu_{\delta^\gamma})$ into $L^\infty$ fails rather dramatically. Indeed, $L^p_\alpha(\mu_{\delta^\gamma})$ is not even embedded in any $\mathfrak{bmo}$ space. The proof is inspired by the proof of~\cite[Theorem 3.3]{PV}. See also~\cite[Appendix B.2.6]{Robinson}.

\begin{proposition}\label{nobmo}
Let $G$ be the ``$ax+b$'' group and $\gamma \in \R\setminus \{ 1\}$. Then
\begin{itemize}
\item[(i)] for every $p\in (1,\infty)$ and $\alpha \geq 0$, $L^p_\alpha(\mu_{\delta^\gamma})$ is not an algebra;
\item[(ii)] for every $p\in (1,\infty)$, $\alpha\geq 0$ and $\eta\in \R$, there exists a function $g\in L^p_\alpha(\mu_{\delta^\gamma})$ which does not belong to $\mathfrak{bmo}(\mu_{\delta^\eta})$. In particular, $L^p_\alpha(\mu_{\delta^\gamma})$ does not embed in any $\mathfrak{bmo}(\mu_{\delta^\eta})$.
\end{itemize}
\end{proposition}

\begin{proof}
Let $\psi$, $\phi$, $\tilde \phi$ be functions on the real line such that
\[ \psi\in C^\infty_c (0,1), \qquad \psi\geq 0, \qquad \psi=1 \mbox{ on } [1/4,3/4],\] 
\[\phi\in C^\infty_c (-1,1), \qquad \phi\geq 0, \qquad \phi=1 \mbox{ on } [0,1/2],\]
\[\widetilde \phi\in C^\infty(\R), \qquad \widetilde \phi \geq 0, \qquad \widetilde \phi=0 \mbox{ on } [0,1/2], \qquad \widetilde \phi =1 \mbox{ on } [1,\infty).\]
Let $\nu \in \R$. Define
\begin{equation}\label{g-controesem-def}
g_\nu(x,a) = \psi(x/a) \phi(a) a^{-\nu}, \qquad \widetilde g_\nu(x,a) = \psi(x/a) \widetilde \phi(a) a^{-\nu}.
\end{equation}
We shall prove that
\begin{itemize}
\item if $\gamma<1$ and $\nu \in \left(\frac{1-\gamma}{2p}, \frac{1-\gamma}{p}\right)$, then $g_\nu\in L^p_k(\mu_{\delta^\gamma})$ for every $k\in \N$, $g_\nu^2\notin L^p(\mu_{\delta^\gamma})$, $g_\nu\notin \mathfrak{bmo}(\mu_{\delta^\eta})$ for every $\eta\in \R$;
\item if $\gamma>1$ and $\nu \in \left(\frac{1-\gamma}{p}, \frac{1-\gamma}{2p}\right)$, then $\widetilde g_\nu\in L^p_k(\mu_{\delta^\gamma})$ for every $k\in \N$, $\widetilde g_\nu^2\notin L^p(\mu_{\delta^\gamma})$, $\widetilde g_\nu\notin \mathfrak{bmo}(\mu_{\delta^\eta})$ for every $\eta\in \R$.
\end{itemize}
It is proved in~\cite[Theorem 3.3]{PV} that if $\nu \in \left(\frac{1}{2p}, \frac{1}{p}\right)$, then $g_\nu \in L^p_k$ for every $k\in \N$, but $g_\nu^2 \notin L^p$ and $g_\nu\notin L^\infty$.

\smallskip

Let $\gamma<1$. It is shown in~\cite[Proof of Theorem 3.3]{PV} that $X_0^{k_0}X_1^{k_1}g_\nu $ is a sum of terms of the form
\[
\widetilde\psi(x/a) \phi(a) a^{- \nu} +\widetilde\Psi(x,a)
\]
for some $\widetilde\psi\in C^\infty_c  (0,1)$ and $\widetilde\Psi\in C^\infty_c(G)$. From this, it is not hard to see that if $\nu \in \left(\frac{1-\gamma}{2p}, \frac{1-\gamma}{p}\right)$, then $g_\nu \in L^p_k(\mu_{\delta^\gamma})$ for every $k$ but $g_\nu^2\notin L^p(\mu_{\delta^\gamma})$.

We now show that $g_\nu \notin \mathfrak{bmo}(\mu_{\delta^\eta})$, independently of $\eta \in \R$: more precisely, for every $\eta \in \R$, there exists a family of $\mathfrak{h}^1(\mu_{\delta^\eta})$-atoms $A_y$, $y\in (0,1)$, such that
\begin{equation}\label{atomsnobmo}
\sup_y \int_G g_\nu \, A_y \, \dd\mu_{\delta^\eta} =\infty.
\end{equation}
The conclusion will then follow by~\cite[Theorem 2]{MedaVolpi}.

Let $R\subset G$ be the rectangle $[-1,1] \times [1/2,3/2]$ and split it into 
\[
R= R_+ \cup R_-, \qquad R_- = [-1,0]\times  [1/2,3/2], \quad R_+ = [0,1]\times  [1/2,3/2].
\]
It is not hard to see that the function
\[
A\coloneqq \mathbf{1}_{R_+} - \mathbf{1}_{R_-}
\]
is, up to a constant, an $\mathfrak{h}^1(\mu_{\delta^\eta})$-atom supported in the ball $B_s$ for $s>0$ large enough (recall Subsection~\ref{Subsec_HardyBMO}). Then, for $y\in (0,1)$ define
\[
A_y(x,a)\coloneqq A((0,y)^{-1}(x,a))\delta(0,y)^{1-\eta} =  A(y^{-1}x, y^{-1}a) y^{\eta-1}.
\]
It is easy to see that for every $y$ the function $A_y$ is a multiple of an $\mathfrak{h}^1(\mu_{\delta^\eta})$-atom supported in $B((0,y),s)$, with constant independent of $y$. Then
\begin{align*}
\int_{\R}\int_{\R_+} g_\nu (x,a)\, A_y(x,a)\, a^{-\eta} \, \frac{\dd x \, \dd a }{a}
& =y^{\eta-1} \int_{y/2}^{3y/2} a^{-\eta-\nu-1} \phi(a) \int_0^y \psi(x/a)\, \dd x \, \dd a \\
& =y^{\eta-1} \int_{y/2}^{3y/2} a^{-\eta-\nu} \phi(a) \int_0^{y/a} \psi(t)\, \dd t \, \dd a
\end{align*}
and since $\phi \approx 1$ in a (right) neighbourhood of $0$, and $\int_0^{y/a} \psi(t)\, \dd t  \approx 1$ since $y/a \approx 1$ for $a\in [y/2, 3y/2]$, if $y$ is small enough
\begin{align*}
\int_{\R}\int_{\R_+} g_\nu (x,a)\, A_y(x,a)\, a^{-\eta} \frac{\dd x \, \dd a }{a}
& \approx y^{\eta-1} \int_{y/2}^{3y/2} a^{-\eta-\nu} \, \dd a \approx y^{-\nu}.
\end{align*}
Since $\nu>0$,  $y^{-\nu} \to +\infty$ when $y\to 0$, which proves~\eqref{atomsnobmo} and all that concerns $\gamma<1$.

\smallskip

The case $\gamma>1$ is similar, and we omit the proof. One just has to replace $\phi$ with $\tilde \phi$, observe that $\nu<0$ and let $y\to +\infty$.
\end{proof}

\section{Applications to PDEs}\label{Sec_applications}
In this section, we shall prove local well-posedness results for nonlinear heat and Schr\"odinger equations. We shall follow a similar approach to that of~\cite[Section 6]{BBR}. 

If $\alpha \geq 0$, we say that a smooth function $A\colon \R^m \to \R$ is \emph{$\alpha$-admissible at 0} if 
\[\partial_1^{h_1}\dots\; \partial_m^{h_m} A(0,\dots,0)=0 \quad \mbox{whenever} \quad \sum_k h_k \leq [\alpha].\]
Moreover, we define 
\[Y^p_\alpha(\mu_\chi) \coloneqq L^p_\alpha(\mu_\chi) \cap L^\infty\]
to be the Banach space endowed with the norm $\| \cdot \|_{Y^p_\alpha(\mu_\chi) } \coloneqq \| \cdot \|_{L^p_\alpha(\mu_\chi)} + \| \cdot \|_\infty$. 
\begin{lemma}\label{lemmaG}
Let $\alpha \geq 0$. Let $G\colon \R^2 \to \R$ be smooth and $\alpha$-admissible at 0. For every $R>0$, every $p\in (1,\infty)$ and every $f_1,f_2\in L^p_\alpha(\mu_\chi)\cap L^\infty$ with $\|f_j\|_\infty \leq R$, $j=1,2$,
\[
\| G(f_1,f_2)\|_{L^p_\alpha(\mu_\chi)} \lesssim_R \| f_1\|_{L^p_\alpha(\mu_\chi)} + \| f_2\|_{L^p_\alpha(\mu_\chi)}.
\]
\end{lemma}
\begin{proof}
Let $\alpha \in [0,1)$. Then by Theorem~\ref{Salphaloc}
\begin{align*}
\| G(f_1,f_2)\|_{L^p_\alpha(\mu_\chi)} 
& \approx  \|S^{{\loc}}_{\alpha}G(f_1,f_2)\|_{L^p(\mu_\chi)}+\|G(f_1,f_2)\|_{L^p(\mu_\chi)}\\
& \lesssim  \|S^{{\loc}}_{\alpha}f_1\|_{L^p(\mu_\chi) }+ \|S^{{\loc}}_{\alpha}f_2\|_{L^p(\mu_\chi)} + \| f_1\|_{L^p(\mu_\chi)} + \| f_2\|_{L^p(\mu_\chi)}\\
& \approx \| f_1\|_{L^p_\alpha(\mu_\chi)} + \| f_2\|_{L^p_\alpha(\mu_\chi)}
\end{align*}
since $G$ is locally Lipschitz, $G(0,0)=0$ and by the definition of $S^\loc_{\alpha}$. It then remains to prove that, if the statement holds for $\alpha$, then it holds for $\alpha +1$.

Assume, then, that $G$ is $(\alpha+1)$-admissible at 0 and assume that the statement holds for $\alpha$. By Proposition~\ref{alpha+1}, 
\begin{align*}
\| G(f_1,f_2)\|_{L^p_{\alpha+1}(\mu_\chi)} 
&\approx \| G(f_1,f_2)\|_{L^p_\alpha(\mu_\chi)} + \sum_{i=1}^\ell \| X_i G(f_1,f_2)\|_{L^p_\alpha(\mu_\chi)}  \\
& \lesssim_R \| f_1\|_{L^p_\alpha(\mu_\chi)} + \| f_2\|_{L^p_\alpha(\mu_\chi)} +  \sum_{i=1}^\ell \| X_i G(f_1,f_2)\|_{L^p_\alpha(\mu_\chi)}.
\end{align*}
Since for every vector field $X_i$
\[
X_i G(f_1,f_2) = (\partial_1 G)(f_1,f_2) X_i f_1 + (\partial_2 G)(f_1,f_2) X_i f_2
\]
it will be enough to prove that 
\begin{equation}\label{XG}
\| (\partial_j G)(f_1,f_2) X_i f_j\|_{L^p_\alpha(\mu_\chi)} \lesssim_R  \| f_1\|_{L^p_{\alpha+1}(\mu_\chi)} +  \| f_2\|_{L^p_{\alpha+1}(\mu_\chi)}
\end{equation}
when $j=1,2$. We shall prove it when $j=1$, for the case $j=2$ is analogous.

Define $r$ and $q$ such that 
\[
\frac{1}{r}= \frac{\alpha}{(\alpha+1)p}, \qquad \frac{1}{q} = \frac{1}{(\alpha +1)p}
\]
and apply Theorem~\ref{algebraproduct} with $q_1=r$, $p_2=p$, $q_2=\infty$, $p_1=q$, so that
\begin{align*}
&\| (\partial_1 G)(f_1,f_2) X_i f_1\|_{L^p_\alpha(\mu_\chi)} \\
&\qquad\lesssim \|(\partial_1 G)(f_1,f_2)\|_{L^r_\alpha(\mu_\chi)} \| X_i f_1\|_{L^q(\mu_\chi)} + \| (\partial_1 G)(f_1,f_2)\|_\infty \| X_i f_1\|_{L^p_\alpha(\mu_\chi)}\\
& \qquad \lesssim (\| f_1\|_{L^r_\alpha(\mu_\chi)}+\| f_2\|_{L^r_\alpha(\mu_\chi)})  \| f_1\|_{L^q_1(\mu_\chi)} + c(R)\| f_1\|_{L^p_{\alpha+1}(\mu_\chi)},
\end{align*}
the second line for the inductive assumption (observe that $\partial_1 G$ is $\alpha$-admissible at 0), by the continuity of $\partial_1 G$ on the compact $[0,R]\times [0,R]$ and Proposition~\ref{alpha+1}. Observe now that~\eqref{XG} will be proved if 
\begin{equation}\label{XGbis}
(\| f_1\|_{L^r_\alpha(\mu_\chi)}+\| f_2\|_{L^r_\alpha(\mu_\chi)}) \| f_1\|_{L^q_1(\mu_\chi)}   \lesssim_R  \| f_1\|_{L^p_{\alpha+1}(\mu_\chi)} +  \| f_2\|_{L^p_{\alpha+1}(\mu_\chi)},
\end{equation}
and this is what we shall do in the remaining part of the proof. 

Apply Proposition~\ref{mean} with $\varepsilon= \alpha+1$, $\beta = 0$, $\alpha=\alpha$ and $\theta = \alpha/(\alpha+1)$. Then for $j=1,2$
\[
\| f_j\|_{L^r_\alpha(\mu_\chi)} \lesssim \| f_j\|^{\alpha/ (\alpha+1)}_{L^p_{\alpha+1}(\mu_\chi)} \|f_j\|_\infty ^{1/(\alpha+1)} \lesssim_R \| f_j\|^{\alpha/ (\alpha+1)}_{L^p_{\alpha+1}(\mu_\chi)} ,
\]
while the same proposition with $\varepsilon= \alpha+1$, $\beta = 0$, $\alpha=1$ and $\theta = 1/(\alpha+1)$ gives 
\[
\| f_1\|_{L^q_1(\mu_\chi)} \lesssim \| f_1\|^{1/ (\alpha+1)}_{L^p_{\alpha+1}(\mu_\chi)}  \|f_1\|_\infty ^{\alpha/(\alpha+1)} \lesssim_R  \| f_1\|^{1/ (\alpha+1)}_{L^p_{\alpha+1}(\mu_\chi)} .
\]
Therefore for every $j=1,2$
\[
\| f_j\|_{L^r_\alpha(\mu_\chi)} \| f_1\|_{L^q_1(\mu_\chi)}  \lesssim_R \| f_j\|^{\alpha/ (\alpha+1)}_{L^p_{\alpha+1}(\mu_\chi)} \| f_1\|^{1/ (\alpha+1)}_{L^p_{\alpha+1}(\mu_\chi)}  \lesssim_R \| f_j\|_{L^p_{\alpha+1}(\mu_\chi)} +  \| f_1\|_{L^p_{\alpha+1}(\mu_\chi)},
\]
and this proves~\eqref{XGbis}, thus~\eqref{XG}, thus the statement.
\end{proof}

\begin{lemma}\label{lem:Flip}
Let $\alpha\geq 0$ and let $F\colon \R \to \R$ be smooth and $(\alpha+1)$-admissible at 0. Then for every $R>0$ there exists $c(R)$ such that 
\[
\|F(u)-F(v)\|_{Y^p_\alpha(\mu_\chi)} \leq c(R) \, \|u-v\|_{Y^p_\alpha(\mu_\chi)}
\]
for every $u,v\in Y^p_\alpha(\mu_\chi)$ such that $\|u\|_{Y^p_\alpha(\mu_\chi)}\leq R$, $\|v\|_{Y^p_\alpha(\mu_\chi)}\leq R$.
\end{lemma}

\begin{proof} Since $F$ is locally Lipschitz we obviously have 
\[
\|F(u)-F(v)\|_{\infty}\leq \|u-v\|_{\infty}\leq \|u-v\|_{Y^p_\alpha(\mu_\chi)}.
\]
To prove the estimate on the $L^p_\alpha(\mu_\chi)$ norm, instead, write 
\begin{equation}\label{FG}
F(x)-F(y)=(y-x)G(x,y), \qquad G(x,y)=\int_0^1F'(x+\tau(y-x))\, \dd \tau.
\end{equation} 
Observe that by the assumption on $F$, $G$ is $\alpha$-admissible at 0. Hence $\|G(u,v)\|_{L^p_\alpha(\mu_\chi)}\leq c_R$ by Lemma~\ref{lemmaG}. Moreover, $G(u,v)$ is bounded since $G$ is continuous and $u,v$ are bounded. Therefore, by Theorem~\ref{algebraproduct} applied to~\eqref{FG}
\begin{align*}
\|F(u)-F(v)\|_{L^p_\alpha(\mu_\chi)}
&\lesssim \|u-v\|_{\infty}\|G(u,v)\|_{L^p_\alpha(\mu_\chi)}+\|u-v\|_{L^p_\alpha(\mu_\chi)}\|G(u,v)\|_{\infty}\\
& \lesssim_R \|u-v\|_{Y^p_\alpha(\mu_\chi)}.
\end{align*}
This completes the proof.
 \end{proof}
\subsection{The nonlinear heat equation associated with $\Delta_\chi$}
Consider the Cauchy problem
\begin{equation}\label{heat}
\begin{cases}
&\hspace{-0.3cm}\partial_t u+\Delta_\chi u=F(u)\\
& \hspace{-0.3cm}u(0,\,  \cdot\,)=u_0.
\end{cases}
\end{equation}
\begin{theorem}\label{nonlinearheat}
Let $\chi$ be a positive character of $G$, $\alpha\geq 0$ and let $F\colon \R \to \R$ be $(\alpha+1)$-admissible at 0. Let $\tau>0 $, $I=[0,\tau]$ and $R>0$ be such that $c(R)\leq 1/(2\tau)$, where $c(R)$ is that of Lemma~\ref{lem:Flip}. Then, for every $u_0\in Y^p_\alpha(\mu_\chi)$ such that $\|u_0\|_{Y^p_\alpha(\mu_\chi)}\leq R/2$, there exists a unique solution $u\in C^0_I  Y^p_\alpha(\mu_\chi)$ of the Cauchy problem~\eqref{heat} such that 
\[
\|u\|_{C^0_I Y^p_\alpha(\mu_\chi)}\lesssim \|u_0\|_{ Y^p_\alpha(\mu_\chi)}.
\]
\end{theorem}
\begin{proof}
By the Duhamel formula, a solution $u$ to \eqref{heat} is given by 
\[
u(t, x)=e^{-t\Delta_\chi}u_0(x)+\int_0^{t} e^{-(t-s)\Delta_\chi}F(u(s,x))\, \dd s. 
\]
Consider the linear operator $D\colon C^0_I Y^p_\alpha(\mu_\chi)\rightarrow C^0_I Y^p_\alpha(\mu_\chi)$ given by 
\[
D f(t,x)=\int_0^{t} e^{-(t-s)\Delta_\chi} f(s,x)\,   \dd s.
\]
Using the properties of the semigroup $e^{-t\Delta_\chi}$, we have that 
\begin{align*} 
\|Df(t, \cdot )\|_{L^p_\alpha(\mu_\chi)} 
& = \| (I+\Delta_\chi)^{\alpha/2} Df(t,\cdot )\|_{L^p(\mu_\chi)} 
\\& \leq  \int_0^t \| (I+\Delta_\chi)^{\alpha/2} e^{-(t-s) \Delta_\chi} f(s, \cdot)  \|_{L^p(\mu_\chi)}\, \dd s
\\& \leq t \sup_{s\in I} \|(I+\Delta_\chi)^{\alpha/2} f(s, \cdot )\|_{L^p(\mu_\chi)} 
\\& \leq |I| \, \|f\|_{C^0_I L^p_\alpha(\mu_\chi)},
\end{align*}
and 
\begin{align*} 
\|Df(t, \cdot )\|_{\infty} 
& \leq  \int_0^t \| e^{-(t-s) \Delta_\chi} f(s, \cdot)  \|_{\infty}\, \dd s
\\& \leq |I| \, \|f\|_{C^0_I L^\infty}.
\end{align*}
so that $\| D\|_{C^0_I Y^p_\alpha(\mu_\chi)\to C^0_I Y^p_\alpha(\mu_\chi)}\leq |I|$. Moreover, by Lemma~\ref{lem:Flip}
\[
\|F(u)-F(v)\|_{C^0_IY^p_\alpha(\mu_\chi)} \leq c(R) \, \|u-v\|_{C^0_IY^p_\alpha(\mu_\chi)} \leq \frac{1}{2|I|} \|u-v\|_{C^0_IY^p_\alpha(\mu_\chi)}
\]
for every $u,v$ such that $\|u\|_{C^0_IY^p_\alpha(\mu_\chi)}, \|v\|_{C^0_IY^p_\alpha(\mu_\chi)} \leq R$. Finally, by~\eqref{equivNorms}, since the semigroup $e^{-t\Delta_\chi}$ commutes with any power of $I+\Delta_\chi$ and by its contractivity on $L^q(\mu_\chi)$ for every $q\in [1,\infty]$, we have 
\[
\|e^{-t\Delta_\chi} u_0\|_{Y^p_\alpha(\mu_\chi)}\leq \|u_0\|_{Y^p_\alpha(\mu_\chi)}\leq R/2.
\]
The abstract iteration argument~\cite[Proposition 1.38]{Tao} gives the result. 
\end{proof}
\subsection{The nonlinear Schr\"odinger equation associated with $\mathcal{L}$}
We shall consider the nonlinear Cauchy problem
\begin{equation}\label{Schr}
\begin{cases}
         &\hspace{-0.3cm} i\partial_t u + \mathcal{L} u = F(u)  \\
          &\hspace{-0.3cm} u(0,\, \cdot \, )=u_0.
\end{cases}
\end{equation}
To study the local well-posedness of~\eqref{Schr} we shall need the embedding $L^p_\alpha(\lambda)\hookrightarrow L^\infty$, and this is the reason for restricting to $\mu_\chi= \lambda$ (hence $\Delta_\chi=\mathcal{L}$). The choice of $p=2$ is, instead, due to the fact that $e^{it\mathcal{L}}$ is bounded on $L^2(\lambda)$, but in general is not bounded on $L^p(\lambda)$ if $p\neq 2$.
\begin{theorem}
Let $\chi$ be a positive character of $G$, $\alpha>d/2$, $u_0\in L^2_\alpha(\lambda)$ and let $F\colon \R \to \R$ be $(\alpha+1)$-admissible at 0. Let $\tau>0 $, $I=[0,\tau]$ and $R>0$ be such that $c(R)\leq 1/(2\tau)$, where $c(R)$ is that of Lemma~\ref{lem:Flip}. Then, for every $u_0\in L^{2}_\alpha(\lambda)$ such that $\|u_0\|_{L^2_\alpha(\lambda)}\leq R/2$, there exists a unique solution $u\in C^0_I  L^2_\alpha(\lambda)$ of the Cauchy problem~\eqref{Schr} such that 
\[
\|u\|_{C^0_I L^2_\alpha(\lambda)}\lesssim \|u_0\|_{ L^2_\alpha(\lambda)}.
\]
\end{theorem}

\begin{proof} By the Duhamel formula, a solution $u$ to \eqref{Schr} is given by 
\[
 u(t,x) = e^{it\mathcal{L}}u_0(x) -i\int_0^t e^{i(t-s) \mathcal{L}} F(u(s,x))\, \dd s.
\]
Consider the linear operator $D\colon C^0_I L^2_\alpha(\lambda)\rightarrow C^0_I L^2_\alpha(\lambda)$ given by 
\[
D f(t,x)= -i\int_0^t e^{i(t-s) \mathcal{L}} f(s,x) \, \dd s.
\]
By the boundedness of the semigroup $e^{it\mathcal{L}}$ on $L^2(\lambda)$, one can prove as before that
\begin{align*} 
\|Df(t, \cdot )\|_{L^2_\alpha(\lambda)} 
& \leq |I| \, \|f\|_{C^0_I L^2_\alpha(\lambda)}
\end{align*}
so that $\| D\|_{C^0_I L^2_\alpha(\lambda)\to C^0_I L^2_\alpha(\lambda)}\leq |I|$. Moreover, by Lemma~\ref{lem:Flip} (observe that $Y^2_\alpha(\lambda) = L^2_\alpha(\lambda)$ by the Sobolev embedding)
\[
\|F(u)-F(v)\|_{C^0_I L^2_\alpha(\lambda)} \leq c(R) \, \|u-v\|_{C^0_IL^2_\alpha(\lambda)} \leq \frac{1}{2|I|} \|u-v\|_{C^0_I L^2_\alpha(\lambda)}
\]
for every $u,v$ such that $\|u\|_{C^0_IL^2_\alpha(\lambda)}, \|v\|_{C^0_IL^2_\alpha(\lambda)} \leq R$. Finally, by~\eqref{equivNorms}, since the semigroup $e^{it\mathcal{L}}$ commutes with any power of $I+\mathcal{L}$ and by its boundedness (it is unitary, indeed) on $L^2(\lambda)$ we have 
\[
\|e^{it\mathcal{L}} u_0\|_{L^2_\alpha(\lambda)}\leq \|u_0\|_{L^2_\alpha(\lambda)}\leq R/2.
\]
The abstract iteration argument~\cite[Proposition 1.38]{Tao} gives the result. 
\end{proof}

\section{Proof of Theorem~\ref{teoLocalRiesz}: boundedness of the local Riesz Transforms} \label{Sec_Proof_Riesz}
This final section is dedicated to the proof of Theorem~\ref{teoLocalRiesz}. We begin with some lemmata which will be of use. The first is rather elementary and its proof is omitted.
\begin{lemma}\label{lemma_convoluzione}
Let $g \in L^1(\mu_\chi) \cap L^1(\rho)$. Then the operator $f \mapsto f * g$ is bounded on $L^p(\mu_\chi)$ for every $1\leq p \leq \infty$, with norm bounded by $\|g\|_{L^1(\mu_\chi)}^{1/p} \|g\|_{L^1(\rho)}^{1/p'}$.
\end{lemma}

\begin{lemma}\label{lemmaQt}
Let $B=B(c_B,r_B)$ be a ball of radius $r_B\leq 1$, $t\in [r_B^2,1]$ and $y,z\in B$.  Let 
\[Q_t^\chi(x,y,z) = P_t^\chi(x,y)-P_t^\chi(x,z).\]
There exists $\epsilon\in (0,1)$ and $c_1>0$ such that for every $\beta<2c_1$ and $J\in \mathfrak{l}^m$
\[
\int_{2r_B \leq |c_B^{-1}x| \leq 2}|X_{J,x} Q_t^\chi(x,y,z)|^2 e^{\beta |c_B^{-1}x|^2/t} \, \dd \mu_\chi(x) \lesssim \left( \frac{d_C(y,z)}{\sqrt{t}}\right)^{2\epsilon} t^{-m} \chi^{-1}(c_B)\delta(c_B) V(\sqrt{t})^{-1}.
\]
\end{lemma}
\begin{proof}
For every $x\in G$, the function $(t,y) \xmapsto{u} P_t^\chi(x,y)$ is a solution of the heat equation $\partial_t u +\Delta_\chi u=0$.  Therefore, also $(t,y)\xmapsto{u_J} X_{J,x} P_t^\chi (x, y)$ is a solution of $\partial_t u +\Delta_\chi u=0$. Thus, by~\cite[Proposition 3.2]{SaloffCoste} (it can be applied by~\cite[eq.\ (2.4)]{Varopoulos1})
\begin{align*}
|X_{J,x} Q_t^\chi(x,y,z)| \lesssim \left( \frac{d_C(y,z)}{\sqrt{t}}\right)^{\epsilon} \sup_{(\tau, w)\in Q} X_{J,x} P_\tau^\chi (x,w)
\end{align*}
for some $\epsilon >0$, where $Q=\left(\frac{4}{9} t,\frac{20}{9} t\right )\times B\left(y,\frac{4}{3}\sqrt t\right)$. By means of~\eqref{Ptgamma_ptgamma},~\eqref{Xdelta}, Lemma~\ref{estimates_ptgamma}~(iii) and the assumptions on $r_B, t, y, z$ we get
\begin{align*}
X_{J,x} P_\tau^\chi(x,w)
& \lesssim \chi^{-1}(w)\delta(w) \chi^{-1/2}(w^{-1}x) \tau^{-(d+m)/2} e^{-b|w^{-1}x|^2/\tau}
\end{align*}
so that (observe that $\tau \approx t$, $t^{-d/2}\approx V(\sqrt{t})^{-1}$ and $\delta(w)\approx \delta(c_B)$, $\chi(w)\approx \chi(c_B)$)
\begin{align}\label{XJqPuntuale}
|X_{J,x} Q_t^\chi(x,y,z)|  \lesssim \left( \frac{d_C(y,z)}{\sqrt{t}}\right)^{\epsilon} \chi^{-1/2}(c_B)\delta(c_B)\chi^{-1/2}(x) t^{-m/2}V(\sqrt{t})^{-1} e^{-c_1 |c_B^{-1}x|^2/t}
\end{align}
for some $c_1>0$. Therefore
\begin{align*}
&\int_{2r_B \leq |c_B^{-1}x| \leq 2}|  X_{J,x} Q_t^\chi(x,y,z)|^2 e^{\beta |c_B^{-1}x|^2/t} \, \dd \mu_\chi(x)  \\ 
&\qquad \qquad \lesssim \left( \frac{d_C(y,z)}{\sqrt{t}}\right)^{2\epsilon} t^{-m} V(\sqrt{t})^{-2}\chi^{-1}(c_B)\delta^{2}(c_B)   \int_{2r_B \leq |c_B^{-1}x| \leq 2} e^{(-2c_1+\beta)|c_B^{-1}x|^2/t}\, \dd \rho(x) \\ 
 &\qquad \qquad= \left(\frac{d_C(y,z)}{\sqrt{t}}\right)^{2\epsilon} t^{-m} V(\sqrt{t})^{-2}\chi^{-1}(c_B)\delta(c_B)  \int_{2r_B \leq |v| \leq 2} e^{(-2c_1+\beta)|v|^2/t}\, \dd \rho(v).
\end{align*}
It remains then to notice that
\[
\int_{2r_B \leq |v| \leq 2} e^{(-2c_1+\beta)|v|^2/t}\, \dd \rho(v) \lesssim V(\sqrt{t})
\]
as in~\cite[p.\ 9]{PV}.
\end{proof}

We are now in the position to prove Theorem~\ref{teoLocalRiesz}, whose proof occupies the remainder of the paper.
\begin{proof}[Proof of Theorem~\ref{teoLocalRiesz}]
To the sake of clarity, we split the proof in various steps.

\smallskip

\noindent \textbf{Step 1. } We prove the $L^2(\mu_\chi)$-boundedness of $X_J (\Delta_\chi +cI)^{-m/2}$. 

By the spectral theorem and~\eqref{unitary_equivalence}
\begin{align*}
X_J(\Delta_\chi +c I)^{-m/2} 
&= X_J\left(\chi^{-1/2}(\Delta +c I)^{-m/2}\chi^{1/2}\right) \\
& = \sum_{ J\in \mathfrak{l}^n, \, n\leq m } c_I \chi^{-1/2} X_J (\Delta +c I)^{-m/2}\chi^{1/2} \\
& = \sum_{ J\in \mathfrak{l}^n, \, n\leq m }  c_I \chi^{-1/2} X_J(\Delta +c I)^{-n/2} (\Delta +c I)^{(-m+n)/2}\chi^{1/2},
\end{align*}
for some constants $c_J$. By~\cite[Theorem 4.8, IV]{TER1}, if $c>0$ is large enough then the local Riesz transforms $X_J (\Delta +c I)^{-n/2} $ are bounded on $L^2(\rho)$ for every $I\in \mathfrak{l}^n$; moreover, the operators $(\Delta +c I)^{(-m+n)/2}$ are bounded on $L^2(\rho)$ for every $n\leq m$ and every $c>0$. Thus Step 1 is proved.

\smallskip

Let now $k_{J}^c$ be the convolution kernel of $X_J (\Delta_\chi +cI)^{-m/2}$, and $\phi \in C_c^\infty(B_1)$ be such that $\phi=1 $ on $B_{1/2}$ and $0\leq \phi \leq 1$. Then
\begin{align}
k_{J}^c(x) 
&= c(J) \int_0^\infty t^{m/2-1} e^{-ct} X_J p_t^\chi(x)\, \dd t \nonumber  \\
&= c(J) \left[ \phi(x)\int_0^1 \dots \, \dd t + (1-\phi(x)) \int_0^1 \dots \, \dd t + \int_1^\infty \dots \, \dd t  \right] \label{KJvari} \\
& \eqqcolon k^0_l(x) + k^0_g (x) + k^\infty(x).\nonumber
\end{align}

\smallskip

\noindent \textbf{Step 2.} We prove that the operators $f\mapsto f * k^\infty$ and $f\mapsto f * k^0_g$ are bounded on $L^p(\mu_\chi)$ for every $p\in [1,\infty]$. 

By Lemma~\ref{estimates_ptgamma}~(iii),
\begin{equation}\label{kinfty}
|k^\infty(x)|\leq \chi^{-1/2}(x) \int_1^\infty e^{-c't} e^{-b|x|^2/t}\, \dd t
\end{equation}
for some $c'>0$. For every $k\geq 1$, denote with $A^k_t$ the annulus $B_{2^k\sqrt{t}} \setminus B_{2^{k-1}\sqrt{t}} $. By Lemma~\ref{estimates_ptgamma}~(ii)
\begin{align*}
\|k^\infty\|_{L^1(\rho)} 
&\lesssim \int_1^\infty \int_{B_{\sqrt{t}}} \chi^{-1/2}(x)e^{-c't} e^{-b|x|^2/t} \, \dd \rho(x)\, \dd t \\ & \hspace{4cm} + \sum_{k=1}^\infty \int_1^\infty \int_{A^k_t}\chi^{-1/2}(x)e^{-c't} e^{-b|x|^2/t} \, \dd \rho(x)\, \dd t \\
&\lesssim \int_1^\infty e^{-c' t} e^{(D+\|X\| ) \sqrt{t}} \, \dd t + \sum_{k=1}^\infty \int_1^\infty e^{-c' t - b2^{2k} + (D+ \|X\|) 2^k \sqrt{t}}\, \dd t \lesssim 1,
\end{align*}
if $c>0$ is sufficiently large. As for $k^0_g$, again by Lemma~\ref{estimates_ptgamma}~(ii),
\begin{equation}\label{k0g}
|k^0_g(x)| \lesssim \chi^{-1/2}(x) e^{-b'|x|^2}
\end{equation}
whence
\begin{align*}
\|k^0_g\|_{L^1(\rho)} 
&\lesssim \sum_{k=1}^\infty \int_{A^k_1}  \chi^{-1/2} (x)e^{-b'|x|^2}\, \dd \rho(x)\lesssim \sum_{k=1}^\infty  e^{-b' 2^{2k}} e^{(D+ \|X\| )2^k} \lesssim 1.
\end{align*}
By~\eqref{kinfty} and~\eqref{k0g}, we have
\[
|k^\infty(x)\chi(x)|\leq\chi^{1/2}(x) \int_1^\infty e^{-c't} e^{-b|x|^2/t}\, \dd t, \qquad |k^0_g(x)\chi(x)| \lesssim \chi^{1/2}(x) e^{-b'|x|^2}.
\]
Arguing as above, one can show that $\|k^\infty\|_{L^1(\mu_\chi)}= \|k^\infty \chi\|_{L^1(\rho)}\lesssim 1 $ and that $\|k^0_g\|_{L^1(\mu_\chi)} = \|k^0_g \chi \|_{L^1(\rho}\lesssim 1$. Thus, by Lemma~\ref{lemma_convoluzione} the first step is proved. \emph{A fortiori}, this implies that the operators $f\mapsto f * k^\infty$ and $f\mapsto f * k^0_g$ are bounded from $\mathfrak{h}^1(\mu_\chi) $ to $L^1(\mu_\chi)$ and from $L^\infty(\mu_\chi)$ to $\mathfrak{bmo}(\mu_\chi)$.

\smallskip

\noindent \textbf{Step 3.} We prove that the operator $f\mapsto f * k^0_l $ is bounded from $L^\infty(\mu_\chi)$ to $\mathfrak{bmo}(\mu_\chi)$.

In view of~\cite[Theorem 8.2]{CMM1} together with~\cite[Proposition 4.5, (ii)]{CMM2} and the Step~1, it is enough to prove that, if $K^0_l$ is the integral kernel of the operator $f\mapsto f * k^0_l$, i.e. 
\[K^0_l(x,y)= \chi^{-1}(y)\delta(y) k^0_l(y^{-1}x),\] 
then
\begin{equation}\label{bmoLinfty}
\sup_{B\in \mathcal B_1} \sup_{y,z\in B} \int_{(2B)^c}|K^0_l(y,x)-K^0_l(z,x)| \dd \mu_\chi(x) < \infty
\end{equation} 
and
\begin{equation}\label{bmoLinftyr1}
\sup_{y\in G} \int_{(B(y,2))^c}|K^0_l(y,x)| \dd \mu_\chi (x) < \infty.
\end{equation} 
Since $\phi$ is supported in $B_1$, $k^0_l(x^{-1}y)=0$ if $x\in B(y,2)^c$, so that~\eqref{bmoLinftyr1} is immediate. As for~\eqref{bmoLinfty},
\begin{align}\label{K0ellbmo}
\int_{(2B)^c}|K^0_l(y,x)-K^0_l(z,x)| \dd \mu_\chi(x) \nonumber
&=\int_{(2B)^c}| k^0_l(x^{-1}y)\delta(x)  -   k^0_l(x^{-1}z)\delta(x) |\, \dd \rho(x) \nonumber \\
&=\int_{2r_B<|v|\leq 2}|k^0_l(vc_B^{-1}y)-k^0_l(vc_B^{-1}z)\,|\, \dd \rho(v) \nonumber \\ 
 &\lesssim d_C(y,z)\sum_{j=1}^\ell \int_{2r_B\leq |v|\leq 2} |X_j k^0_l(v)|\, \dd\rho(v).
\end{align}
It remains then to estimate $|X_j k^0_l|$. By the definition of $k^0_l$ and Lemma~\ref{estimates_ptgamma}~(iii), if $|v|\leq 2$ then
\begin{align*}
|X_j k^0_l(v)|
& \lesssim \int_0^1 t^{m/2-1} t^{-d/2-m/2} e^{-b|v|^2/t}\, \dd t + \int_0^1 t^{m/2-1} t^{-d/2-(m+1)/2} e^{-b|v|^2/t}\, \dd t \\
& =  \int_{|v|^2}^\infty \left( \frac{|v|^2}{s}\right)^{-d/2} e^{-b s}\, \dd s + \int_{|v|^2}^\infty \left( \frac{|v|^2}{s}\right)^{-(d+1)/2} e^{-b s}\, \dd s \lesssim |v|^{-d-1},
\end{align*} 
so that by~\eqref{K0ellbmo}
\[
\int_{(2B)^c}|K^0_l(y,x)-K^0_l(z,x)|	\, \dd \mu_\chi(x) \lesssim r_B \int_{2r_B\leq |v|\leq 2} |v|^{-d-1}\, \dd\rho(v).
\]
If $j_0$ is the biggest integer such that $2^{j_0-1}\leq 2r_B$, then $\{v\colon 2r_B\leq |v| \leq 2\} \subseteq \{v\colon 2^{j_0-1}\leq |v| \leq 2\}$ and hence
\[
r_B \int_{2r_B\leq |v|\leq 2} |v|^{-d-1}\, \dd\rho(v) \lesssim r_B \sum_{j=j_0}^1 \int_{A^j_1} |v|^{-d-1}\, \dd\rho(v)\lesssim  r_B \sum_{j=j_0}^1 2^{-j(d+1)} 2^{jd}\, \dd\rho(v)\lesssim 1
\]
which proves~\eqref{bmoLinfty}, and the Step 3.
\smallskip

\noindent \textbf{Step 4.} We prove that the operator $f\mapsto f * (k^0_l + k^0_g)$ is bounded from $\mathfrak{h}^1(\mu_\chi)$ to $L^1(\mu_\chi)$.

Let $k^0\coloneqq k^0_l + k^0_g $. Again by~\cite[Theorem 8.2]{CMM1} together with~\cite[Proposition 4.5, (ii)]{CMM2} and the Step~1, it is enough to prove that, if $K^0$ is the integral kernel of the operator $f\mapsto f * k^0$, i.e. 
\begin{align}\
K^0(x,w) = \chi^{-1}(w)\delta(w)k^0(w^{-1}x) =  \int_0^1 t^{m/2-1} e^{-ct} X_{J,x} P_t^\chi(x,w)\, \dd t \label{kernelK02}
\end{align}
then
\begin{equation}\label{h1L1}
\sup_{B\in \mathcal B_1} \sup_{y,z\in B} \int_{(2B)^c}|K^0(x,y)-K^0(x,z)|\, \dd \mu_\chi(x) < \infty
\end{equation} 
and
\begin{equation}\label{h1L1r1}
\sup_{y\in G} \int_{(B(y,2))^c}|K^0(x,y)| \, \dd \mu_\chi (x) < \infty.
\end{equation} 
As for~\eqref{h1L1r1}, observe that by~\eqref{kernelK02}
\begin{align*}
\int_{(B(y,2))^c}|K^0(x,y)| \, \dd \mu_\chi (x) &= \int_{d_C(x,y)\geq 2} |(k^0\chi)(y^{-1}x)| \delta(y) \, \dd \rho(x) \\
&= \int_{|v|\geq 2} |(k^0_g \chi)(v)| \, \dd \rho(v)\lesssim 1,
\end{align*}
since we proved in the Step 2 that $k^0_g \in L^1(\mu_\chi)$. Thus, only~\eqref{h1L1} is left. We split the integral
\[
\int_{(2B)^c}|K^0(x,y)-K^0(x,z)| \,\dd \mu_\chi(x) = \int_{2r_B \leq |c_B^{-1}x|\leq 2} \dots \, \dd \mu_\chi + \int_{|c_B^{-1}x|>2} \dots \, \dd \mu_\chi = \mathsf{I} + \mathsf{II}
\]
and observe that by~\eqref{kernelK02}
\begin{align*}
\mathsf{II} 
& \leq \int_{|c_B^{-1}x|>2} | (k^0\chi)(y^{-1}x) \delta(y)|\, \dd \rho(x) +  \int_{|c_B^{-1}x|>2} | (k^0\chi)(z^{-1}x) \delta(z)|\, \dd \rho(x)\\ 
& = 2 \int_{|v|>2} | (k^0\chi)(v)|\, \dd \rho(v)\lesssim 1
\end{align*}
as above. As for $\mathsf{I}$, to shorten the notation we denote by $A_B$ the annulus (depending on $B$) $\{x\colon 2r_B\leq d_C(x,c_B)\leq 2\} = B(c_B,2)\setminus B(c_B, 2r_B)$ and we split $\mathsf{I}$ in turn. Indeed, by~\eqref{kernelK02}
\begin{align*}
\mathsf{I}
& \leq \int_0^1 t^{m/2-1} \int_{A_B} | X_{J,x}P_t^\chi(x,y) - X_{J,x}P_t^\chi(x,z)|\, \dd\mu_\chi(x)\, \dd t\\
& = \int_{0}^{r_B^2} \int_{A_B}|Q_t^\chi(x,y,z)|\, \dd \mu_\chi(x) \, \dd t + \int_{r_B^2}^1 \int_{A_B}|Q_t^\chi(x,y,z)|\, \dd \mu_\chi(x) \, \dd t = \mathsf{I}_1 +\mathsf{I}_2.
\end{align*}
As for $\mathsf{I}_1$, by~\eqref{Ptgamma_ptgamma} and the left invariance of the vector fields 
\begin{equation*}
X_{J,x}P_t^\chi(x,w) = \chi^{-1}(w)\delta(w) (X_{J}p_t^\chi)(w^{-1}x)
\end{equation*}
so that $\mathsf{I}_1$ is controlled by
\begin{align}\label{I1}
\int_0^{r_B^2} t^{m/2-1} \int_{A_B} \left(| (\chi^{-1}\delta)(y) (X_{J}p_t^\chi)(y^{-1}x)|+ |(\chi^{-1}\delta)(z)  (X_{J}p_t^\chi)(z^{-1}x)|\right)\, \dd\mu_\chi(x)\, \dd t
\end{align}
and by using Lemma~\ref{estimates_ptgamma}~(iii)
\begin{align*}
&\int_0^{r_B^2} t^{m/2-1} \int_{2r_B \leq |c_B^{-1}x|\leq 2} | (\chi^{-1}\delta)(y)  (X_{J}p_t^\chi)(y^{-1}x)|\, \dd\mu_\chi(x)\, \dd t \\
&\qquad \qquad \qquad \leq \int_0^{r_B^2} t^{-d/2-1} \int_{A_B} \chi^{1/2}(y^{-1}x)\delta(y) e^{-b|y^{-1}x|^2/t}\, \dd \rho(x)\, \dd t \\ 
&\qquad \qquad \qquad \leq \int_0^{r_B^2} t^{-d/2-1}  \int_{r_B \leq |v|\leq 4}e^{-b|v|^2/t}\, \dd \rho(v)\, \dd t
\end{align*}
the second inequality by the change of variables $y^{-1}x=v$ and since the character $\chi$ and $\delta$ are bounded in $B_4$. This can be seen to be finite exactly as in~\cite[p.\ 13]{PV} and the second term in~\eqref{I1} (depending on $z$ instead of $y$) can be treated in the same way, since $y$ plays no role in the estimate. Thus, $\mathsf{I}_1\lesssim 1$.

It remains to estimate $\mathsf{I}_2$. We apply Cauchy-Schwarz inequality to its inner integral, and get
\begin{align*}
& \int_{A_B} | X_{J,x}Q_t^\chi(x,y,z)|\, \dd\mu_\chi(x)\\ 
& \qquad \qquad \leq \left( \int_{A_B} | X_{J,x}Q_t^\chi(x,y,z)|^2 e^{\beta |c_B^{-1}x|^2/t}\, \dd\mu_\chi(x)\right)^{1/2} \left( \int_{A_B} e^{-\beta |c_B^{-1}x|^2/t}\, \dd\mu_\chi(x)\right)^{1/2}\\ 
&\qquad \qquad= \mathsf{I}_{2,1}(t) + \mathsf{I}_{2,2}(t)
\end{align*}
where $\beta\in (0,2c_1)$ and $c_1$ is that of Lemma~\ref{lemmaQt}. To estimate $\mathsf{I}_{2,2}(t)$, when $2r_B\leq \sqrt t$ we choose $j_0$ as the smallest integer such that $2^{j_0+1}\sqrt t\geq 2$, and make the change of variables $c_B^{-1}x=v$. We obtain  
\begin{align*}
\mathsf{I}_{2,2}(t)^2&\leq\chi(c_B)\delta^{-1}(c_B)\int_{2r_B\leq |v|\leq 2}e^{-\beta |v|^2/t}\, \dd \rho(v)\\
&=\chi(c_B)\delta^{-1}(c_B)\left(  \int_{2r_B\leq |v|\leq \sqrt t}e^{-\beta |v|^2/t}\, \dd \rho(v) +\sum_{j=1}^{j_0}  \int_{2^j\sqrt t\leq |v|\leq 2^{j+1}\sqrt t}e^{-\beta |v|^2/t}\, \dd \rho(v)\right)\\
&\lesssim \chi(c_B)\delta^{-1}(c_B)\left(  V(\sqrt t)e^{-\beta r_B^2/t}+\sum_{j=1}^{j_0}  (2^j\sqrt t)^d e^{-\beta 2^j}\right) \lesssim  \chi(c_B)\delta^{-1}(c_B)   V(\sqrt t)e^{-\beta r_B^2/t}\,.
\end{align*}
The estimate when $\sqrt t\leq 2r_B$ is similar and yields as well $\mathsf{I}_{2,2}(t)^2\lesssim\chi(c_B)\delta^{-1}(c_B)   V(\sqrt t)e^{-\beta r_B^2/t}$. By Lemma~\ref{lemmaQt}, moreover,
\begin{align*}
\mathsf{I}_{2,1}(t)^2 \lesssim \left( \frac{r_B}{\sqrt{t}}\right)^{2\epsilon} t^{-m} \chi^{-1}(c_B)\delta(c_B) V(\sqrt{t})^{-1}.
\end{align*}
Therefore
\[
\mathsf{I}_2 \lesssim \int_{r_B^2}^1 t^{m/2-1}\mathsf{I}_{2,1}(t) \cdot \mathsf{I}_{2,2}(t) \, \dd t \lesssim \int_{r_B^2}^1 t^{-1} e^{-(\beta/2) r_B^2/t}\left( \frac{r_B}{\sqrt{t}}\right)^{\epsilon}  \, \dd t\lesssim 1
\]
and this completes the proof of~\eqref{h1L1} and of the Step 4.

\smallskip

\noindent \textbf{Step 5.} We may now conclude the proof. By Steps~2-4, $X_J(\Delta_\chi +cI)^{-m/2}$ is bounded from $\mathfrak{h}^1(\mu_\chi)$ to $L^1(\mu_\chi)$ and from $L^\infty$ to $\mathfrak{bmo}(\mu_\chi)$. Theorem~\ref{teoLocalRiesz} follows by interpolation~\cite[Theorem 5]{MedaVolpi}.
\end{proof}

\section{Final remarks}\label{Sec:finrem}
To conclude, we would like to mention a few natural questions related to the subject of this paper. Indeed, it may be worthy to study embedding theorems and algebra properties of other function spaces -- such as homogeneous Sobolev spaces and Besov spaces associated with the operator $\Delta_\chi$ -- in the spirit of~\cite{CRTN} and~\cite{GallagherSire, FMV} respectively. Moreover, it would be interesting to apply the algebra properties of the Sobolev spaces to other nonlinear differential equations associated with $\Delta_\chi$, such as the wave or the Schr\"odinger equations, and study local or global well-posedness results of associated nonlinear Cauchy problems.

\end{document}